\documentclass[12pt]{amsart}
\usepackage{amsmath,amsthm,amsfonts,amssymb,latexsym,verbatim,bbm}
\input xy
\xyoption{all}
\usepackage{hyperref,multirow}
\usepackage[shortlabels]{enumitem}
\usepackage{young}
\usepackage{tikz, ifthen}
\usepackage[nomessages]{fp}
\usepackage{mathtools}
\usetikzlibrary{calc}
\usetikzlibrary{chains}
\usepackage{booktabs,caption}

\newcounter{x}
\newcounter{y}
\newcounter{z}

\newcommand*\cubecolors[1]{%
  \ifcase#1\relax
  \or\colorlet{cubecolor}{cyan}%
  \or\colorlet{cubecolor}{green}%
  \or\colorlet{cubecolor}{yellow}%
  \or\colorlet{cubecolor}{pink}%
  \or\colorlet{cubecolor}{purple}%
  \or\colorlet{cubecolor}{blue}%
  \else
    \colorlet{cubecolor}{white}%
  \fi
}
\newcommand\yaxis{180}
\newcommand\zaxis{-27}
\newcommand\xaxis{90}

\newcommand\topside[3]{
  \fill[fill=cubecolor, draw=black,shift={(\xaxis:#1)},shift={(\yaxis:#2)},
  shift={(\zaxis:#3)}] (0,0) -- (1,0) -- (0.5,0.25) --(-0.5,0.25)--(0,0);
}

\newcommand\leftside[3]{
  \fill[fill=cubecolor, draw=black,shift={(\xaxis:#1)},shift={(\yaxis:#2)},
  shift={(\zaxis:#3)}] (0,0) -- (0,-1) -- (-0.5,-0.75) --(-0.5,0.25)--(0,0);
}

\newcommand\rightside[3]{
  \fill[fill=cubecolor, draw=black,shift={(\xaxis:#1)},shift={(\yaxis:#2)},
  shift={(\zaxis:#3)}] (0,0) -- (1,0) -- (1,-1) --(0,-1)--(0,0);
}

\newcommand\cube[3]{
  \topside{#1}{#2}{#3} \leftside{#1}{#2}{#3} \rightside{#1}{#2}{#3}
}

\newcommand\planepartition[2]{
 \setcounter{x}{0}
 \foreach \a in {#2} {
    \addtocounter{x}{1}
    \setcounter{y}{-1}
    \cubecolors{\value{x}}
    \foreach \b in \a {
      \addtocounter{y}{1}
      \setcounter{z}{-1}
      \foreach \c in {0,...,\b} {
        \addtocounter{z}{1}
      \ifthenelse{\c=0}{\setcounter{z}{-1},\addtocounter{y}{0}}{
        \FPeval{\newz}{clip(0.55*\the\value{z})}
        \cube{\value{x}}{\value{y}}{\newz};
        \FPeval{\result}{clip(#1-\the\value{y}+2*\the\value{z})}
        \draw[draw=black,shift={(\xaxis:\value{x})},shift={(\yaxis:\value{y})},
  shift={(\zaxis:\newz)}] (0.5,-0.5) node {\textsf{\result}};}
      }
    }
  }
}

\newcommand\planepartitionD[2]{
 \setcounter{x}{0}
 \foreach \a in {#2} {
    \addtocounter{x}{1}
    \setcounter{y}{-1}
    \cubecolors{\value{x}}
    \foreach \b in \a {
      \addtocounter{y}{1}
      \setcounter{z}{-1}
      \foreach \c in {0,...,\b} {
        \addtocounter{z}{1}
      \ifthenelse{\c=0}{\setcounter{z}{-1},\addtocounter{y}{0}}{
        \FPeval{\newz}{clip(0.55*\the\value{z})}
        \cube{\value{x}}{\value{y}}{\newz};
        \FPeval{\result}{clip(#1-\the\value{y}-2*\the\value{z})}
        \draw[draw=black,shift={(\xaxis:\value{x})},shift={(\yaxis:\value{y})},
  shift={(\zaxis:\newz)}] (0.5,-0.5) node {\textsf{\result}};}
      }
    }
  }
}

\newcommand{\mlabel}[1]{\(s_{\mathrlap{#1}}\)}
\newcommand{\dnode}[2][chj]{\node[#1,label={below:\mlabel{#2}}] {};}
\newcommand{\dydots}{ \node[chj,draw=none,inner sep=1pt] {\dots};}
\newcommand{\dnodedub}[2][chj, join=by {<->, double}]{\node[#1,label={below:\mlabel{#2}}] {};}
\newcommand{\dnodebr}[1]{\node[chj,label={below right:\mlabel{#1}}] {};}

\tikzset{node distance=2em, ch/.style={circle,draw,on chain,inner sep=2pt},chj/.style={ch,join}, line width=1pt,baseline=-1ex}

\allowdisplaybreaks[2] \textwidth15.1cm \textheight22cm \headheight12pt \oddsidemargin.4cm
\theoremstyle{plain}
\newtheorem{theorem}{Theorem}[section]
\newtheorem{thm}[theorem]{Theorem}
\newtheorem{lemma}[theorem]{Lemma}
\newtheorem{proposition}[theorem]{Proposition}
\newtheorem{prop}[theorem]{Proposition}
\newtheorem{example}[theorem]{Example}
\newtheorem{definition}[theorem]{Definition}
\newtheorem{defn}[theorem]{Definition}

\newtheorem{corollary}[theorem]{Corollary}
\newtheorem{cor}[theorem]{Corollary}

\newtheorem{rmk}[theorem]{Remark}

\newcommand{\arr}{\rightarrow}

\newcommand{\Z}{\mathbb{Z}}

\newcommand{\mcD}{\mathcal{D}}

\newcommand{\mcG}{\mathcal{G}}

\newcommand{\mcI}{\mathcal{I}}

\newcommand{\mcH}{\mathcal{H}}

\DeclareMathOperator{\adj}{adj}

\DeclareMathOperator{\supp}{S}

\DeclareMathOperator{\flip}{flip}
\DeclareMathOperator{\bp}{bp}
\newcommand{\Cat}{\textsf{Cat}}
\newcommand{\cc}{\textsf{c}}
\newcommand{\SQ}{\sqrt{1-4t}}

\begin{document}

\title[Staircase diagrams and enumeration]{Staircase diagrams and enumeration
of smooth Schubert varieties}

\author{Edward Richmond}
\email{edward.richmond@okstate.edu}

\author{William Slofstra}
\email{weslofst@uwaterloo.ca}

\begin{abstract}
    We enumerate smooth and rationally smooth Schubert varieties in the
    classical finite types $A$, $B$, $C$, and $D$, extending Haiman's
    enumeration for type $A$. To do this enumeration, we introduce
    a notion of staircase diagrams on a graph. These combinatorial
    structures are collections of steps of irregular size, forming interconnected staircases
    over the given graph. Over a Dynkin-Coxeter graph, the set of
    ``nearly-maximally labelled'' staircase diagrams is in bijection with the
    set of Schubert varieties with a complete Billey-Postnikov (BP) decomposition.
    We can then use an earlier result of the authors showing that all
    finite-type rationally smooth Schubert varieties have a complete
    BP decomposition to finish the enumeration.

\end{abstract}
\maketitle
\section{Introduction}

Let $G$ be a simple Lie group over an algebraically closed field and fix a
Borel subgroup $B\subseteq G$. The Schubert varieties $X(w)$ in the flag
variety $G/B$ are indexed by the Weyl group $W$ of $G$. A natural question to
ask is: when is $X(w)$ (rationally) smooth? Many different answers have been
given to this question. In particular, the Lakshmibai-Sandhya theorem states
that a Schubert variety $X(w)$ of type $A$ is smooth if and only if the
permutation $w$ avoids $3412$ and $4231$. There is an analogous pattern
avoidance criteria for classical types due to Billey \cite{Bi98}, and a
root-system pattern avoidance criteria for all finite types due to
Billey-Postnikov \cite{BP05}. Other characterizations of (rationally) smooth
Schubert varieties include the regularity of the Bruhat graph \cite{Ca94},
triviality of the Kazhdan-Lusztig polynomials \cite{De77,KL79} and
palindromicity of the Poincar\'{e} polynomial \cite{Mc77}. A survey of these
results can be found in \cite{BL00}.

Although these criteria allow us to efficiently recognize (rationally) smooth
Schubert varieties, they do not allow us to enumerate such Schubert varieties.
In type $A$ the generating series for the number of smooth Schubert is known, and is
due to Haiman \cite{Ha92} \cite{Bo98}. An alternative (and in fact the first
published) derivation of this generating series is given in \cite{BMB07} by
Bousquet-M\'elou and Butler. In an earlier paper \cite{RS14}, we discuss the
possibility of listing all smooth and rationally smooth Schubert varieties
using Billey-Postnikov decompositions, an idea that goes back to \cite{BP05}.
The purpose of this paper is to complete this idea by enumerating smooth and
rationally smooth Schubert varieties in the finite classical types $A$, $B$,
$C$, and $D$. Specifically, define generating series
\begin{equation*}
    A(t):=\sum_{n=0}^{\infty} a_n\, t^n,\qquad B(t):=\sum_{n=0}^{\infty} b_n\, t^n,\qquad C(t):=\sum_{n=0}^{\infty} c_n\, t^n,
\end{equation*}
\begin{equation*}
    D(t):=\sum_{n=3}^{\infty} d_n\, t^n,\qquad \text{ and }BC(t):=\sum_{n=0}^{\infty} bc_n\, t^n,
\end{equation*}
where the coefficients $a_n,b_n,c_n,d_n$ denote the number of smooth Schubert
varieties of types $A_n, B_n, C_n$ and $D_n$ respectively, and $bc_n$ denotes
the number of rationally smooth Schubert varieties of type $B_n$. Since the
Weyl groups of type $B_n$ and $C_n$ are isomorphic, and $X(w)$ is rationally smooth
in type $B$ if and only if $X(w)$ is rationally smooth in type $C$, we refer to
this last case as ``type $BC$''. For simply-laced types $A$ and $D$, Peterson's
theorem states that a Schubert variety is rationally smooth if and only if it
is smooth \cite{CK03}, so these generating series cover all classes of smooth
and rationally smooth Schubert varieties in finite classical type. The
main result of this paper is:

\begin{theorem}\label{T:main_gen_series}
    Let $W(t) := \sum_n w_n\, t^n$ denote one of the above generating series,
    where $W = A$, $B$, $C$, $D$, or $BC$. Then
    \begin{equation*}
        W(t)=\frac{P_W(t)+Q_W(t)\SQ}{(1-t)^2(1-6t+8t^2-4t^3)}
    \end{equation*}
    where $P_W(t)$ and $Q_W(t)$ are polynomials given in Table \ref{TBL:polys}.
\end{theorem}

\begin{table}[h]
    \centering
    \begin{tabular}{ccc}
    \toprule
        Type  & $P_W(t)$  & $Q_W(t)$  \\
        \midrule
        $A$& $(1-4t)(1-t)^3$              & $t(1-t)^2$     \\
        $B$& $(1-5t+5t^2)(1-t)^3 $        & $(2t-t^2)(1-t)^3$     \\
        $C$& $1-7t+15t^2-11t^3-2t^4+5t^5$ & $t-t^2-t^3+3t^4-t^5 $   \\
        $D$& $(-4t+19t^2+8t^3-30t^4+16t^5)(1-t)^2$   & $(4t-15t^2+11t^3-2t^5)(1-t)$   \\
        $BC$& $1-8t+23t^2-29t^3+14t^4$    & $2t-6t^2+7t^3-2t^4$  \\
    \bottomrule
    \end{tabular}
    \caption{Polynomials in Theorem \ref{T:main_gen_series}.}
    \label{TBL:polys}
\end{table}
Table \ref{TBL:count} gives the number of smooth and rationally smooth Schubert varieties in each type
for rank $n \leq 8$.
\begin{table}[h]
    \begin{tabular}{cccccc}
        \toprule
         & $a_n$  & $b_n$ & $c_n$ & $d_n$ & $bc_n$ \\
        \midrule
        $n=1$& 2      & 2     & 2     &       & 2 \\
        $n=2$& 6      & 7     & 7     &       & 8 \\
        $n=3$& 22     & 28    & 28    & 22    & 34 \\
        $n=4$& 88     & 116   & 114   & 108   & 142 \\
        $n=5$& 366    & 490   & 472   & 490   & 596 \\
        $n=6$& 1552   & 2094  & 1988  & 2164  & 2530 \\
        $n=7$& 6652   & 9014  & 8480  & 9474  & 10842 \\
        $n=8$& 28696  & 38988 & 36474 & 41374 & 46766 \\
        \midrule
    \end{tabular}
    \caption{Number of smooth and rationally smooth Schubert varieties in ranks $n \leq 8$.
    By convention, $A_1 = B_1 = C_1$, and $D_3 = A_3$.}
    \label{TBL:count}
\end{table}
One observation that can be made from Theorem \ref{T:main_gen_series} is that
\begin{equation*}
    A(t)+BC(t)=B(t)+C(t).
\end{equation*}
We give a geometric explanation for this fact in Remark \ref{R:BC_smooth}.

It is well known that the growth of the coefficients of a generating series is
controlled by the singularity of smallest modulus \cite[Theorem IV.7]{FS09}.
For each generating series $W(t)$ in Theorem \ref{T:main_gen_series}, the
smallest singularity is the root
\begin{equation*}
    \alpha:=\frac{1}{6}\left(4-\sqrt[3]{17+3\sqrt{33}}+\sqrt[3]{-17+3\sqrt{33}}\right)\approx 0.228155
\end{equation*}
of the polynomial $1-6t+8t^2-4t^3$ appearing in the denominator. Thus we get an
asymptotic formula for $w_n$ as an immediate corollary of Theorem
\ref{T:main_gen_series} and \cite[Theorem IV.10]{FS09}.
\begin{corollary}
    Let $W(t) = \sum w_n \, t^n$, where $W = A$, $B$, $C$, $D$, or $BC$.
    Then
    \begin{equation*}
        w_n\sim \frac{W_{\alpha}}{\alpha^{n+1}},
    \end{equation*}
    where $W_{\alpha}$ is a constant defined by
    \begin{equation*}
        W_{\alpha}:=\lim_{t\rightarrow \alpha}\,(\alpha-t)\, W(t).
    \end{equation*}
    Table \ref{TBL:constant} gives the approximate value of $W_{\alpha}$ in
    each type.
\end{corollary}
In particular,
\begin{equation*}
    \lim_{n\rightarrow \infty}\frac{w_{n+1}}{w_n}=\alpha^{-1}\approx 4.382985,
\end{equation*}
meaning that the number of (rationally) smooth Schubert varieties $w_n$ grows
at the same rate for every finite classical Lie type.
Interestingly, a
similar ratio was observed for the number of smooth Schubert varieties in types
$E_6$, $E_7$, and $E_8$ by Carrell and Kuttler \cite{CK03}. Their observation
was one of the motivations for our investigation.

\begin{table}[h]
    \begin{tabular}{cccccc}
    \toprule
         & $A$  & $B$ & $C$ & $D$ & $BC$ \\
    \midrule
    $W_{\alpha}$& $0.045352$ & $0.062022$ & $0.057301$ & $0.067269$ & $0.073972$ \\
    \bottomrule
    \end{tabular}
    \caption{Initial constant for the asymptotic number of Schubert varieties by type.}
    \label{TBL:constant}
\end{table}

\subsection{A brief overview of staircase diagrams}
Haiman's original derivation of the generating series $A(t)$ uses a result of
Ryan \cite{Ry87}, which states that smooth Schubert varieties in type $A$ can
be expressed as iterated fibre bundles of Grassmannian flag varieties. Our
proof of Theorem \ref{T:main_gen_series} uses a similar iterated fibre bundle
structure on (rationally) smooth finite type Schubert varieties recently proved
by the authors \cite{RS14}. To study these fibre bundle structures
combinatorially, we introduce a new data structure called a \emph{staircase
diagram}, which lies over the Dynkin diagram of the Weyl group. To start with, we define staircase diagrams over an arbitrary graph. Informally, a staircase diagram is
a collection of connected ``blocks'' of vertices of a graph, where the blocks
are allowed to overlap each other, forming arrangements which resemble
staircases with steps of irregular length, as shown in the picture below (see
example \ref{Ex:typeA_11}):
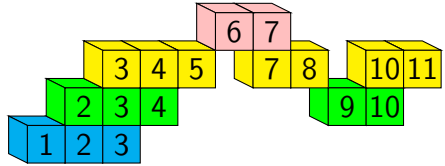
\begin{equation*}
    \begin{tikzpicture}[scale=0.5]
        \planepartition{11}{
            {0,0,0,0,0,0,0,0,1,1,1},
            {0,1,1,0,0,0,0,1,1,1},
            {1,1,0,1,1,0,1,1,1},
            {0,0,0,0,1,1}}
    \end{tikzpicture}
\end{equation*}
These arrangements must satisfy a number of conditions---for instance, no block
can contain another, nor are ``vertical gaps'' allowed---so the following diagram
is \emph{not} a staircase diagram:
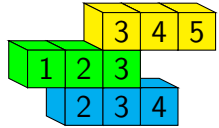
\begin{equation*}
    \begin{tikzpicture}[scale=0.5]
        \planepartition{5}{
            {0,1,1,1,0},
            {0,0,1,1,1},
            {1,1,1,0,0}}
    \end{tikzpicture}
\end{equation*}
The above examples are diagrams over a path; with more complicated graphs we
get more complicated behaviour, such as the example below (on the right) over a
star graph (on the left):
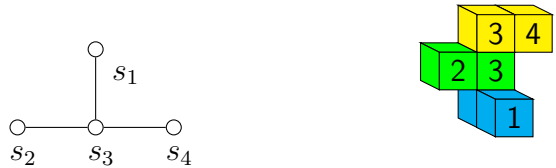
\begin{center}
\begin{tabular}{cc}
    \begin{tikzpicture}
        \begin{scope}[start chain]
            \dnode{2} \dnode{3} \dnode{4}
        \end{scope}
        \begin{scope}[start chain=br going above]
            \chainin(chain-2); \dnodebr{1}
        \end{scope}
    \end{tikzpicture}
    \hspace{1in}
    &
    \begin{tikzpicture}[scale=0.5]
        \planepartitionD{4}{ {0,2,0}, {0,1,1}, {1,1,0}}.
    \end{tikzpicture}
    \\
\end{tabular}
\end{center}
If the underlying graph $\Gamma$ is the Dynkin graph of a Weyl group $W$, then
we can label the blocks of a staircase diagram by elements of $W$, and ultimately
assign to each labelled staircase diagram $\mcD$ an element $\Lambda(\mcD) \in W$
such that the Schubert variety $X(\Lambda(\mcD))$ is an iterated fibre bundle.
Combinatorially, the (parabolic) fibre bundle structures on $X(\Lambda(\mcD))$
correspond to upwardly closed subdiagrams of $\mcD$.

As we discuss in Section \ref{SS:bp}, two of the main results of \cite{RS14} are
that (parabolic) fibre bundle structures on $X(w)$ correspond to
Billey-Postnikov decompositions of $w$ and that if $X(w)$ is rationally smooth, such fibre bundle structures always exist.  The main technical result of this paper, stated in
Theorem \ref{T:staircasebij}, is that there is a bijection between
``nearly-maximal labelled staircase diagrams'' over the Dynkin diagram of $W$,
and elements of $W$ with a complete Billey-Postnikov decomposition.  While we focus on finite type Weyl groups, this result applies to all Coxeter groups.

\subsection{Outline}
In the next section, we define staircase diagrams on a general graph. We
then spend the rest of the first part of the paper (sections 3-6) on
staircase diagrams over a Coxeter-Dynkin graph, proving that for finite-type
Dynkin diagrams, labelled staircase diagrams are in bijection with rationally
smooth Schubert varieties. In the second part of the paper (sections 7-10),
we enumerate staircase diagrams over a finite-type Dynkin diagram, proving
Theorem \ref{T:main_gen_series}.

\subsection{Acknowledgements}
The authors thank Sara Billey and Jim Carrell for helpful conversations, and
Sara Billey for providing data on the number of smooth and rationally smooth
Schubert varieties.  The first author would like to thank Anthony Kable for
some helpful discussions on asymptotics.  The second author thanks Mark Haiman
for helpful remarks. The pictures of staircase diagrams are based on TikZ code
for plane partitions by Jang Soo Kim.


\section{Staircase diagrams on graphs}

We start by introducing our main data structure. For this, we use some standard terminology concerning posets and graphs.  Specifically, if $(X,\preceq)$
is a poset, recall that $x' \in X$ covers $x \in X$ if $x' \succ  x$ and there is no
$y \in X$ with $x' \succ  y \succ  x$. A subset $Y \subset X$ is a chain if it is totally
ordered, and saturated if $x' \succ  y \succ  x$ for some $x', x \in Y$ implies that $y
\in Y$.  Given $A, B \subseteq X$, we say that $A \prec B$ if $a \prec b$ for all $a
\in A$, $b \in B$. When the partial order on $X$ is clear, we refer to the
poset $(X,\preceq)$ by $X$. Since we will eventually be working with
Coxeter-Dynkin diagrams, we use $S$ to denote the vertex set of a graph $\Gamma$,
which we fix for this section. Given $s, t \in S$, we write $s \adj t$ to mean
that $s$ is adjacent to $t$, or in other words that there is an edge between
$s$ and $t$ in $\Gamma$. We allow $\Gamma$ to have multiple edges between two
vertices, since Coxeter-Dynkin diagrams can have this property.  However, the
staircase diagrams we introduce in this section only depend on whether $s$ and
$t$ are adjacent, so edge multiplicities will not play a role until we start
working with Coxeter groups. Throughout the paper, we assume that $\Gamma$ does
not have any loops, or in other words that $s$ is never adjacent to
itself. We say that a subset $B \subset S$ is connected if the induced
subgraph with vertex set $B$ is connected, and that $B, B' \subset S$ are
adjacent if some element of $B$ is adjacent to some element of $B'$. Finally,
given a collection $\mcD \subseteq 2^S$ and a vertex $s \in S$, we let
\begin{equation*}
    \mcD_s:=\{B\in\mcD\ |\ s\in B\}.
\end{equation*}

We can now state the main definition of the paper:
\begin{definition}\label{D:staircase}
    Let $\mcD = (\mcD, \preceq)$ be a partially ordered subset of $2^S$ not
    containing the empty set. We say that $\mcD$ is a \emph{staircase diagram}
    if the following are true:
    \begin{enumerate}[(1)]
        \item Every $B\in\mcD$ is connected, and if $B$ covers $B'$ then $B\cup B'$ is connected.
        \item The subset $\mcD_s$ is a chain for every $s \in S$.
        \item If $s\adj t$, then $\mcD_s\cup \mcD_t$ is a chain, and $\mcD_s$ and $\mcD_t$
            are saturated subchains of $\mcD_s \cup \mcD_t$.
        \item If $B\in\mcD$, then there is some $s\in S$ (resp. $s'\in S$) such
            that $B$ is the minimum element of $\mcD_s$ (resp. maximum element of
            $\mcD_{s'}$).
    \end{enumerate}
\end{definition}
This definition is meant to formalize an arrangement of blocks sitting over a
graph, such that the blocks overlap each other in a particular way.  Note that
elements of the set $\mcD$ are called \emph{blocks}.  We now consider some
specific examples illustrating the different parts Definition
\ref{D:staircase}.

\begin{example}\label{Ex:typeA_11}
    The picture
    \begin{equation*}
        \begin{tikzpicture}[scale=0.5]
            \planepartition{11}{
                {0,0,0,0,0,0,0,0,1,1,1},
                {0,1,1,0,0,0,0,1,1,1},
                {1,1,0,1,1,0,1,1,1},
                {0,0,0,0,1,1}}
        \end{tikzpicture}
    \end{equation*}
    represents a staircase diagram $\mcD$ over a simple path with vertices
    $S = \{s_1,\ldots,s_{11}\}$, where $s_i$ is adjacent to $s_{i+1}$.  The elements of
    $\mcD$ correspond to connected blocks of uniform color in this diagram (for notational simplicity, we pictorially label $s_i$ by $i$), so
    \begin{equation*}
        \mcD = \left\{ \{s_1,s_2,s_3\}, \{s_2,s_3,s_4\}, \{s_3,s_4,s_5\}, \{s_6,s_7\}, \{s_7,s_8\}, \{s_9,s_{10}\}, \{s_{10},s_{11}\} \right\}.
    \end{equation*}
    The covering relations for $\mcD$ are given by the vertical adjacencies. In this case,
    $\mcD$ has covering relations
    \begin{align*}
        & \{s_1,s_2,s_3\} \prec \{s_2,s_3,s_4\} \prec \{s_3,s_4,s_5\} \prec \{s_6,s_7\},  \\
        & \{s_9,s_{10}\} \prec \{s_7,s_8\} \prec \{s_6,s_7\}, \text{ and }  \\
        & \{s_9,s_{10}\} \prec \{s_{10},s_{11}\}.
    \end{align*}
\end{example}

\begin{example}
    Let $\Gamma$ be the graph
        
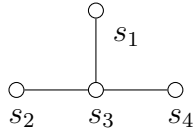
\begin{equation*}
    \begin{tikzpicture}
        \begin{scope}[start chain]
            \dnode{2} \dnode{3} \dnode{4}
        \end{scope}
        \begin{scope}[start chain=br going above]
            \chainin(chain-2); \dnodebr{1}
        \end{scope}
    \end{tikzpicture}
    \end{equation*}
    Then the staircase diagram $\mcD = \left\{ \{s_1,s_3,s_4\}, \{s_2,s_3,s_4\} \right\}$
    with $\{s_1,s_3,s_4\} \prec \{s_2,s_3,s_4\}$ is represented pictorially by
    
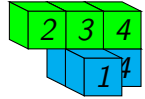
\begin{equation*}
        \begin{tikzpicture}[scale=0.5]
            \planepartitionD{4}{ {1,2,0}, {1,1,1}}.
        \end{tikzpicture}
    \end{equation*}
    The diagram $\mcD = \left\{ \{s_2\}, \{s_4\}, \{s_1,s_3\}\right\}$ with covering relations
    $\{s_2\} \prec \{s_1,s_3\}$ and $\{s_4\} \prec \{s_1,s_3\}$ is represented by
    
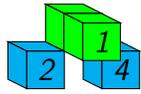
\begin{equation*}
        \begin{tikzpicture}[scale=0.5]
            \planepartitionD{4}{ {1,0,1}, {0,2,0}}
        \end{tikzpicture}.
    \end{equation*}
\end{example}
Part (1) of Definition \ref{D:staircase} states that the block of a
diagram must be a connected subset of the vertices, and that blocks can only
touch if they contain common or adjacent vertices. Part (2) of the definition
states that blocks with a common vertex must be comparable, or in other words
must be stacked one over the other.
\begin{example}
    Let $\mcD = \left\{\{s_1,s_2\}, \{s_2,s_3\} \right\}$. The partial order with
    no relations, which we might picture as
    
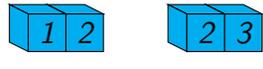
\begin{equation*}
        \begin{tikzpicture}[scale=0.5]
            \planepartition{3}{ {0,1,1}}
        \end{tikzpicture} \quad \quad
        \begin{tikzpicture}[scale=0.5]
            \planepartition{3}{ {1,1}}
        \end{tikzpicture},
    \end{equation*}
    is not a staircase diagram. To get a valid staircase diagram, we must
    have either $\{s_1,s_2\} \prec \{s_2,s_3\}$ or $\{s_2,s_3\} \prec \{s_1,s_2\}$. These two
    partial orders correspond to the pictures
    
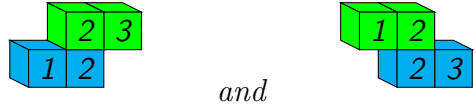
\begin{equation*}
        \begin{tikzpicture}[scale=0.5]
            \planepartition{3}{ {0,1,1}, {1,1}}
        \end{tikzpicture}
        \qquad \text{ and } \qquad
        \begin{tikzpicture}[scale=0.5]
            \planepartition{3}{ {1,1}, {0,1,1}}
        \end{tikzpicture}
    \end{equation*}
    respectively.
\end{example}


The following example show violations of parts (3) and (4) of Definition \ref{D:staircase}:

\begin{example}
    Consider the diagrams:
    \begin{center}
    \begin{tabular}{cc}
        \begin{tikzpicture}[scale=0.5]
            \planepartitionD{5}{ {0,1,2},{1},{0,1,1,1}}
        \end{tikzpicture} \hspace{1in}
        &
        \begin{tikzpicture}[scale=0.5]
            \planepartition{3}{ {1,1},{1,1,1}}
        \end{tikzpicture}
    \end{tabular}
    \end{center}
    The first diagram violates part (3) of Definition \ref{D:staircase} since the chain
    $\mcD_{s_4}$ is not a saturated subchain of $\mcD_{s_4} \cup \mcD_{s_5}$.  The second diagram violates part (4) of Definition \ref{D:staircase} since the block $\{s_2,s_3\}$ is not maximal
    in $\mcD = \mcD_{s_2} = \mcD_{s_3}$.
\end{example}

Given $J \subset S$, we define
\begin{equation*}
    \mcD_J:=\{B\in\mcD\ |\ J\subseteq B\}.
\end{equation*}
We record some immediate consequences of the axioms which will be helpful in the sequel.
\begin{lemma}\label{L:staircaseprops}
    Let $\mcD$ be a staircase diagram. Then:
    \begin{enumerate}[(a)]
        \item $\mcD_J$ is a chain for every $J \subset S$.
        \item If $B, B' \in \mcD$, then $B \not\subseteq B'$.
        \item If $B, B' \in \mcD$ either contain a common element, or
            are adjacent, then $B$ and $B'$ are comparable in the partial order on $\mcD$.
    \end{enumerate}
\end{lemma}
\begin{proof}
    $\mcD_J$ is the intersection of $\mcD_s$, $s \in J$, so part (a) follows from
    the fact that each $\mcD_s$ is a chain.

    For part (b), find $s$ and $s'$ such that $B$ is the maximal and minimal block
    of $\mcD_s$ and $\mcD_{s'}$ respectively. Then $B$ is the only block of
    $\mcD_{\{s,s'\}}$, so there cannot be another block $B' \in \mcD$ strictly containing $B$.

    Part (c) follows immediately from conditions (2) and (3) of the definition.
\end{proof}
Definition \ref{D:staircase} is symmetric with respect to reversing the partial order
$\preceq$ on $\mcD$, so we can make the following definition:
\begin{defn}
    If $\mcD$ is a staircase diagram, then $\flip(\mcD)$ is the staircase diagram
    with the reverse partial order.
\end{defn}

To get the pictorial diagram for $\flip(\mcD)$, we simply flip the diagram from
top to bottom.
\begin{example}
    If $\mcD$ is the diagram in Example \ref{Ex:typeA_11}, then $\flip(\mcD)$ is
    \begin{equation*}
        \begin{tikzpicture}[scale=0.5]
            \planepartition{11}{
                {0,0,0,0,1,1,0,0,0,0,0},
                {1,1,0,1,1,0,1,1,1,0,0},
                {0,1,1,0,0,0,0,1,1,1,0},
                {0,0,0,0,0,0,0,0,1,1,1}}
        \end{tikzpicture}
    \end{equation*}
\end{example}

We finish the section with the following simple definitions:
\begin{defn}
    The \emph{support} of a staircase diagram $\mcD$ is the set of vertices
    \begin{equation*}
        \supp(\mcD) := \bigcup_{B \in \mcD} B.
    \end{equation*}
    We say $\mcD$ is \emph{connected} if the support is a connected subset
    of the base graph. A subset $\mcD' \subset \mcD$ is a \emph{subdiagram} if
    $\mcD'$ is a saturated subset of $\mcD$.
\end{defn}
It is easy to see that a subdiagram of a staircase diagram with the induced
partial order is a staircase diagram in its own right.  Every staircase diagram
$\mcD$ is a union of connected subdiagrams supported on the connected
components of $\supp(\mcD)$.
\begin{example}
    The diagram
     
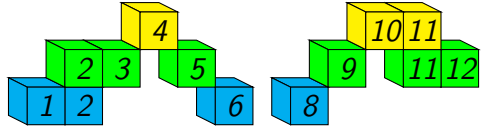
\begin{equation*}
        \begin{tikzpicture}[scale=0.5]
            \planepartition{12}{
                {0,0,0,0,1,0,1,0,0,0,1,1},
                {1,1,0,1,0,0,0,1,0,1,1},
                {0,1,1,0,0,0,0,0,1,0,0}}
        \end{tikzpicture}
    \end{equation*}
    has two connected subdiagrams. The support of the diagram is
    $\{s_1,\ldots,s_6,s_8,\ldots,s_{12}\}$.
\end{example}


\section{Staircase diagrams on Coxeter systems}

A Coxeter group $W$ is a group generated by a finite set of simple generators
$S$, modulo relations $(st)^{m_{st}} = e$ for all $s,t \in S$, where $m_{st}$
is a collection of integers satisfying $m_{st} \geq 2$ for all $s\neq t$ and
$m_{ss} = 1$.   Let $\ell:W\arr\Z_{\geq 0}$ denote the length function and let
$\leq$ denote the Bruhat partial order on $W$.  The pair $(W,S)$ is called a
Coxeter system, and any system is uniquely determined by the multigraph with
vertex set $S$ and $m_{st}-2$ edges between $s,t \in S$. This graph is called
the Coxeter-Dynkin graph (or Dynkin diagram) of the system. A staircase diagram
on $(W,S)$ is simply a staircase diagram on the Coxeter-Dynkin graph.

Given a subset $J \subset S$, let $W_J$ be the parabolic subgroup generated
by $J$.  Let $W^J$ denote the set of minimal length left coset representatives of $W/W_J$ and
${}^J W$ be the set of minimal length right coset representatives of $W_J\backslash W$.  We say that
a subset $J \subseteq S$ is \emph{spherical} if the parabolic subgroup $W_J$ is
finite, in which case we let $u_J$ denote the unique maximal element of $W_J$.
\begin{defn}
    We say that a staircase diagram $\mcD$ is \emph{spherical} if every $B \in
    \mcD$ is spherical.
\end{defn}
Given $J$, every element $w \in W$ can be written uniquely as $w = vu$ for some
$v \in W^J$ and $u \in W_J$. This is called the \emph{parabolic decomposition} of $w$
with respect to $J$. We also let $D_L(w)$ and $D_R(w)$ denote the left and
right descent sets of $w$ respectively, and $S(w)$ be the support of $w$, or in
other words the set of simple reflections which appear in some reduced
expression for $w$. Given an element $w \in W$, we say that a parabolic
decomposition $w = vu$, $v \in W^J$, $u \in W_J$, is a \emph{Billey-Postnikov
(BP) decomposition} if $S(v) \cap J  \subseteq D_L(u)$.

Staircase diagrams can be used to describe iterated Billey-Postnikov
decompositions.  Specifically, given a staircase diagram $\mcD$ on a Coxeter
system $(W,S)$, define functions $J_R, J_L : \mcD \arr S$ by
\begin{align*}
    J_R(B) \cong J_R(B, \mcD) & := \left\{s \in B\ |\ B \neq \min\{\mcD_s\}\right\} \text{ and } \\
    J_L(B) \cong J_L(B, \mcD) & := \left\{s \in B\ |\ B \neq \max\{\mcD_s\}\right\} \\
          & \ = J_R(B, \flip(\mcD)).
\end{align*}
The factors of a Billey-Postnikov decomposition give rise to a labelling of the
staircase diagram:
\begin{defn}\label{D:labelling}
    Let $\mcD$ be a staircase diagram on a Coxeter system $(W,S)$, such that
    the sets $J_L(B)$ and $J_R(B)$ are spherical for all $B \in \mcD$. We
    say that a function
    \begin{equation*}
        \lambda : \mcD \mapsto W
    \end{equation*} is a \emph{labelling} of $\mcD$ if
    \begin{enumerate}[(1)]
        \item $J_R(B) \subseteq D_R(\lambda(B))$,
        \item $J_L(B) \subseteq D_L(\lambda(B))$, and
        \item $S(\lambda(B) u_{J_R(B)}) = S( u_{J_L(B)} \lambda(B)) = B$.
    \end{enumerate}
\end{defn}
Since the definition is symmetric, $w$ is a labelling of $\mcD$ if and only if
\begin{equation*}
    \lambda^{-1} : \flip(\mcD) \mapsto W
\end{equation*}
given by  $B \mapsto \lambda(B)^{-1}$ is a labelling of $\flip(\mcD)$. Furthermore, if $\lambda$ is a labelling of $\mcD$
and $\mcD'$ is a subdiagram, then the restriction $\lambda|_{\mcD'}$ of $\lambda$ to $\mcD'$
is a labelling of $\mcD'$.
\begin{defn}\label{D:maximal_labelling}
    If $\mcD$ is a spherical staircase diagram, then the function $\lambda : \mcD \mapsto W$
    given by $\lambda(B)=u_B$ is a labelling of $\mcD$. We call this the \emph{maximal labelling}.
\end{defn}

\begin{defn}\label{D:labelling2}
    Given a labelling $\lambda$ of a staircase diagram $\mcD$, let $\overline{\lambda}(B)
    := \lambda(B) u_{J_R(B)}$. We set
    \begin{equation}\label{E:labelprod}
        \Lambda(\mcD, \lambda) := \overline{\lambda}(B_n) \overline{\lambda}(B_{n-1}) \cdots \overline{\lambda}(B_1),
    \end{equation}
    where $B_1,\ldots,B_n$ is some linear extension of the poset $\mcD$.
    Usually $\lambda$ is clear, and we write $\Lambda(\mcD)$ in place of
    $\Lambda(\mcD,\lambda)$.
\end{defn}
If $\lambda$ is a labelling, and $B$ and $B'$ are incomparable in $\mcD$, then
$S(\overline{\lambda}(B)) = B$ and $S(\overline{\lambda}(B')) = B'$ are disjoint and
non-adjacent by Lemma \ref{L:staircaseprops} part(c).  In particular, $\overline{\lambda}(B)$ and
$\overline{\lambda}(B')$ commute and thus $\Lambda(\mcD)$ does not depend on the choice of
linear extension. Also,
\begin{equation*}
    J_R(B_i)=B_i \cap (B_{i-1} \cup \cdots \cup B_1)
\end{equation*}
for every $i=1,\ldots,n$, and by definition $\overline{\lambda}(B_i) \in
W^{J_R(B_i)}$, so the product in equation (\ref{E:labelprod}) is reduced in the sense that
\begin{equation*}
    \ell(\Lambda(\mcD))=\ell(\overline{\lambda}(B_n))+\cdots + \ell(\overline{\lambda}(B_1)).
\end{equation*}
Moreover, we have that $S(\Lambda(\mcD)) = \supp(\mcD)$.
\begin{example}\label{Ex:typeAcont}
    The permutation group on $n+1$ letters is the Coxeter group of type $A_n$. The Coxeter-Dynkin diagram of $A_n$ is the simple path of
    length $n$, with vertex set $s_1,\ldots,s_n$, where $s_i$ is the simple
    transposition $(i\ \ i+1)$.
    
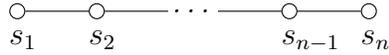
\begin{equation*}
        \begin{tikzpicture}[start chain]
            \dnode{1} \dnode{2} \dydots \dnode{n-1} \dnode{n}
        \end{tikzpicture}
    \end{equation*}
    The staircase diagram $\mcD$ in Example \ref{Ex:typeA_11} can be considered
    as a staircase diagram over the Coxeter group $A_{11}$.  This diagram has a linear extension
    \begin{equation*}
        \{s_1,s_2,s_3\}, \{s_2,s_3,s_4\}, \{s_9,s_{10}\}, \{s_3,s_4,s_5\}, \{s_7, s_8\}, \{s_{10}, s_{11}\},
            \{s_6, s_7\},
    \end{equation*}
    and
    \begin{align*}
        & J_R(\{s_2,s_3,s_4\}) = \{s_2,s_3\}, \quad J_R(\{s_3,s_4,s_5\}) = \{s_3,s_4\}, \\
        & J_R(\{s_6,s_7\}) = \{s_7\}, \quad J_R(\{s_{10},s_{11}\}) = \{s_{10}\}, \text{ and } \\
        & J_R(\{s_1,s_2,s_3\}) = J_R(\{s_7, s_8\}) = J_R(\{s_9, s_{10}\}) = \emptyset.
    \end{align*}
    If $\lambda : \mcD\rightarrow W$ is the maximal labelling, then
    \begin{align*}
        & \overline{\lambda}(\{s_2, s_3, s_4\}) = s_2 s_3 s_4, \quad \overline{\lambda}(\{s_3,s_4,s_5\}) = s_3 s_4 s_5, \\
        & \overline{\lambda}(\{s_6, s_7\}) = s_7 s_6, \quad \overline{\lambda}(\{s_{10}, s_{11}\}) = s_{10} s_{11}, \\
        & \overline{\lambda}(\{s_1, s_2, s_3\}) = s_1 s_2 s_3 s_1 s_2 s_1, \\
        & \overline{\lambda}(\{s_7, s_8\}) = s_7 s_8 s_7, \text{ and }\, \overline{\lambda}(\{s_9, s_{10}\}) = s_9 s_{10} s_9
    \end{align*}
    and
    \begin{equation*}
        \Lambda(\mcD) = (s_7 s_6) (s_{10} s_{11}) (s_7 s_8 s_7) (s_3 s_4 s_5) (s_9 s_{10} s_9) (s_2 s_3 s_4) (s_1 s_2 s_3 s_1 s_2 s_1).
    \end{equation*}
\end{example}

If $\mcD'$ is a subdiagram of $\mcD$, we use the convention that $\Lambda(\mcD') := \Lambda(\mcD',\lambda|_{\mcD'})$.
Also note that if $\mcD'$ is a lower order ideal of $\mcD$, then both $\mcD'$
and $\mcD \setminus \mcD'$ are subdiagrams of $\mcD$.
\begin{prop}\label{P:labelparabolic}
    Let $\lambda$ be a labelling of a staircase diagram $\mcD$. Given a lower order
    ideal $\mcD' \subset \mcD$, set $\mcD'' := \mcD \setminus \mcD'$. If
    $\Lambda(\mcD) = vu$ is the parabolic decomposition of $\Lambda(\mcD)$ with respect
    to $J = \supp(\mcD')$, then $u = \Lambda(\mcD')$ and $v = \Lambda(\mcD'')
    u_K$, where
    \begin{equation*}
        K := \{s \in \supp(\mcD'')\ |\ \min(\mcD''_s) \neq \min(\mcD_s) \}.
    \end{equation*}
    Furthermore, the support of $v$ is
    \begin{equation*}
        S(v) = \bigcup_{B \in \mcD''} B.
    \end{equation*}
\end{prop}
\begin{proof}
    Since $\mcD'$ is a lower order ideal, we can find a linear extension
    $B_1,\ldots,B_n$ of $\mcD$ such that $\mcD' = \{ B_1,B_2,\ldots,B_i\}$ for
    some $i \leq n$. Since $\overline{\lambda}(B_j) \in W^{J_R(B_i)}$, we have that
    \begin{equation*}
        v = \overline{\lambda}(B_n) \cdots \overline{\lambda}(B_{i+1}) \text{ and }
        u = \overline{\lambda}(B_{i}) \cdots \overline{\lambda}(B_1).
    \end{equation*}
    Since $\mcD'$ is a lower order ideal, $J_R(B, \mcD) = J_R(B, \mcD')$ for
    all $B \in \mcD'$, so $\Lambda(\mcD') = \overline{\lambda}(B_{i}) \cdots
    \overline{\lambda}(B_1)$. Also, the support set $S(v)=\bigcup_{B \in \mcD''} B$.

    The calculation for $v$ is more difficult. Observe that $J_R(B,\mcD'')
    \subseteq J_R(\mcD)$ for all $B \in \mcD''$, and that
    \begin{equation*}
        K = \bigcup_{B \in \mcD''} J_R(B, \mcD) \setminus J_R(B, \mcD'').
    \end{equation*}
    For each $i < j \leq n$, set $K_j = J_R(B_j, \mcD) \setminus J_R(B_j,
    \mcD'')$ and suppose $s \in K_j$ for some $j > i$. By definition, $s
    \not\in B_l$ for $i < l < k$, but $s \in B_k$ for some $1 \leq k \leq i$.
    If $t \adj s$ then $\mcD_s$ is a saturated subset of the chain $\mcD_s \cup
    \mcD_t$, so $t \not\in B_l$ for any $k < l < j$. Consequently, if $t \in
    J_R(B_j, \mcD)$, then $t \not\in J_R(B_j, \mcD'')$.  This implies $K_j$ is disjoint and
    non-adjacent to the set
    \begin{equation*}
        B_{i+1} \cup \cdots \cup B_{j-1} \cup J_R(B_j, \mcD'').
    \end{equation*}
    In particular, the sets $K_j$, $i < j \leq n$ are disjoint and
    non-adjacent.  Thus
    \begin{align*}
        v & =  \overline{\lambda}(B_n) \cdots \overline{\lambda}(B_{i+1}) \\
          & = \lambda(B_n) u_{J_R(B_n, \mcD)} \cdots \lambda(B_{i+1}) u_{J_R(B_{i+1}, \mcD)} \\
          & = \lambda(B_n) u_{J_R(B_n, \mcD'')} u_{K_n} \cdots \lambda(B_{i+1}) u_{J_R(B_{i+1}, \mcD'')}
            u_{K_{i+1}} \\
          & = \Lambda(\mcD'') u_{K_n} u_{K_{n-1}} \cdots u_{K_{i+1}} = \Lambda(\mcD'') u_{K}.
    \end{align*}
\end{proof}
This leads to the main theorem of this section, which states that we can
determine the descent sets of $\Lambda(\mcD)$ using only information about $\mcD$ and
``local'' information about each $\lambda(B)$. If $\mcD$ has labelling $\lambda$, we
use the convention that $\Lambda(\flip(\mcD)) := \Lambda(\flip(\mcD), \lambda^{-1})$.
\begin{thm}\label{T:labelling}
    Let $\lambda$ be a labelling of a staircase diagram $\mcD$. Then:
    \begin{enumerate}[(a)]
        \item $\Lambda(\flip(\mcD)) = \Lambda(\mcD)^{-1}$.
        \item $D_R(\Lambda(\mcD)) = \left\{s \in S\ |\ \min(\mcD_s) \preceq \min(\mcD_t) \text{ for all }
            s \adj t \text{ and } s \in D_R\left(\lambda\left(\min(\mcD_s)\right)\right)
            \right\}$.
        \item $D_L(\Lambda(\mcD)) = \left\{s \in S\ |\ \max(\mcD_s) \succeq \max(\mcD_t) \text{ for all }
            s \adj t \text{ and } s \in D_L\left(\lambda\left(\max(\mcD_s)\right)\right)
            \right\}$.
    \end{enumerate}
\end{thm}
To prove the theorem we need the following lemma:
\begin{lemma}[Lemma 5.4 of \cite{RS14}]\label{L:adjacentpara}
    If $w = vu$ is a parabolic decomposition, and $s \in D_L(u)\setminus S(v)$, then $s \in D_L(w)$ if and only if $s$ is not
    adjacent to any element of $S(v)$.
\end{lemma}
\begin{proof}[Proof of Theorem \ref{T:labelling}]
    First we prove part (a) by induction on the number of sets in $\mcD$. If
    $\mcD = \{B\}$, then the proposition follows immediately from the definitions.
    Otherwise, take a maximal block $B_n \in \mcD$, and let $\mcD'$ be the
    lower order ideal $\mcD \setminus \{B_n\}$, so that $\Lambda(\mcD) = \lambda(B_n)
    u_{J_R(B_n,\mcD)} \Lambda(\mcD')$.  By induction, we have $\Lambda(\flip(\mcD')) =
    \Lambda(\mcD')^{-1}$. The set $\mcD'' = \{B_n\}$ is a lower order ideal of
    $\flip(\mcD)$, and
    \begin{equation*}
        K = \{ s \in \supp(\mcD')\ |\ \max(\mcD'_s) \neq \max(\mcD_s) \}
            = \supp(\mcD') \cap B_n = J_R(B_n, \mcD),
    \end{equation*}
    so
    \begin{align*}
        \Lambda(\mcD)^{-1} & = \Lambda(\mcD')^{-1} u_{K} \lambda(B_n)^{-1} = \Lambda(\flip(\mcD'))
            u_K \Lambda(\mcD'') = \Lambda(\flip(\mcD))
    \end{align*}
    by Proposition \ref{P:labelparabolic}. Thus part (a) follows by induction.

    Next, suppose we are given $s \in S$. Let $B_s = \max(\mcD_s)$ and define the lower order ideal
    $$\mcD':=\{B \in \mcD\ |\ B \preceq B_s\}$$ in $\mcD$.
    By Proposition \ref{P:labelparabolic}, $\Lambda(\mcD) = v \Lambda(\mcD')$ where $v
    \in W^J$ and $\Lambda(\mcD') \in W_J$.  By construction, $s \not\in S(v) =
    \bigcup_{B \not\in \mcD'} B$, so by Lemma \ref{L:adjacentpara} we have that
    $s \in D_L(\Lambda(\mcD))$ if and only if $s \in D_L(\Lambda(\mcD'))$ and $s$ is not
    adjacent to any element of $S(v)$. But $\{B_s\}$ is an lower order ideal of
    $\flip(\mcD')$, so once again by Proposition \ref{P:labelparabolic} we see
    that $s \in D_R(\Lambda(\flip(\mcD'))$ if and only if $s \in D_R(\lambda^{-1}(B_s))$.
    Since $\Lambda(\flip(\mcD')) = \Lambda(\mcD')^{-1}$, we conclude that $s \in D_L(\Lambda(\mcD'))$
    if and only if $s \in D_L(\Lambda(B_s))$. Finally, $s$ is adjacent to an element
    of $S(v)$ if and only if there is some $t \in S$ adjacent to $s$ with
    $\max(\mcD_t) \not\in \mcD'$. This latter condition holds if and only if
    $\max(\mcD_t) \succ  B_s$. We conclude that part (c) holds, and part (b)
    follows by combining parts (a) and (c).
\end{proof}
With Theorem \ref{T:labelling} we can make a connection between labelled staircase
diagrams and BP decompositions.
\begin{cor}\label{C:bpstaircase}
    Let $\mcD$ be a staircase diagram with a linear ordering $B_1,\ldots,B_n$,
    and let $\mcD^i$ be the subdiagram $\mcD^i := \{B_1,\ldots,B_{i-1}\}$,
    $i=2,\ldots,n$. If $\lambda$ is a labelling of $\mcD$, then
    \begin{equation*}
        \Lambda(\mcD^{i+1}) = \overline{\lambda}(B_i)\cdot \Lambda(\mcD^i)
    \end{equation*}
    is a BP decomposition with respect to $\supp(\mcD^i)$ for
    every $i=2,\ldots,n$.
\end{cor}
\begin{proof}
    By definition, $$S(\overline{\lambda}(B_i)) \cap S(\mcD^i) = B_i \cap S(\mcD^i) =
    J_R(B_i).$$ So $\overline{\lambda}(B_i) \Lambda(\mcD^i)$ is a BP decomposition if
    and only if $J_R(B_i) \subseteq D_L(\Lambda(\mcD^i))$. Given $s \in J_R(B_i)$,
    let $B_j$ be the predecessor of $B_i$ in the chain $\mcD_s$.
    If $t \in \supp(\mcD^i)$ is adjacent to $s$, then $\mcD_s$ is a saturated
    subset of $\mcD_s \cup \mcD_t$, and consequently $\max(\mcD^i_t) \preceq B_j$.
    Hence Theorem \ref{T:labelling} implies that $s \in D_L(\Lambda(\mcD^i))$ if
    and only if $s \in D_L(\lambda(B_j))$. But $s \in J_L(B_j) \subseteq D_L(\lambda(B_j))$
    by definition, so we conclude that $\overline{\lambda}(B_i) \Lambda(\mcD^i)$ is a BP
    decomposition.
\end{proof}

It is convenient to make the following definitions:
\begin{align*}
    D_R(\mcD)&:= \left\{s\in S\ |\ \min(\mcD_s)\preceq\min(\mcD_t) \text{ for all }s\adj t\right\}. \\
    D_L(\mcD)&:= \left\{s\in S\ |\  \max(\mcD_s)\succeq\max(\mcD_t) \text{ for all }s\adj t\right\}. \\
    & \ =D_R(\flip(\mcD)).
\end{align*}
If $\mcD$ is spherical, then these are the right and left descent sets of $\Lambda(\mcD, \lambda)$,
where $\lambda$ is the maximal labelling of $\mcD$. For a general labelling, $s
\in D_L(\Lambda(\mcD))$ if and only if $s \in D_L(\mcD) \cap D_L(\lambda(\max(\mcD_s))$.
Similarly, with the right descent set we have $s \in D_R(\Lambda(\mcD))$ if and
only if $s \in D_R(\mcD) \cap D_R(\lambda(\min(\mcD_s))$.

\subsection{BP decompositions and the geometry of Schubert varieties}\label{SS:bp}

In the next three sections, we prove that if $G$ is a Lie group of finite type, then staircase diagrams with certain labellings are in bijection with rationally smooth Schubert varieties in $G/B$.  We illustrate this bijection with a motivating example connecting Corollary \ref{C:bpstaircase} to the geometry of Schubert varieties.  For any $J\subseteq S$, let $P_J$ denote the corresponding parabolic subgroup of $G$ and consider the natural projection between flag varieties
$$\pi:G/B\rightarrow G/P_J.$$
The projection induces a $P_J/B$ fibre bundle structure on $G/B$.  For any $w\in W$, the Schubert variety $X(w):=\overline{BwB}/B\subseteq G/B$.  If $w=vu$ is the parabolic decomposition of $w$ with respect to $J$, then the restriction of $\pi$ to $X(w)$ gives the projection
$$\pi:X(w)\rightarrow X^J(v):=\overline{BvP_J}/P_J.$$  The following theorem, proved in \cite{RS14}, is the main connection between the geometry of Schubert varieties and BP decompositions.

\begin{theorem}\label{T:bpgeom}
The parabolic decomposition $w=vu$ is a BP decomposition if and only if the map $\pi$ induces a $X(u)$-fibre bundle structure on $X(w)$.
\end{theorem}

\begin{example}\label{Ex:typeA_fibration}
Let $\mcD$ be the staircase diagram of type $A_4$ given by the following picture
  
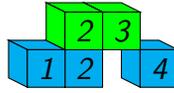
\begin{equation*}
        \begin{tikzpicture}[scale=0.5]
            \planepartition{4}{{1,0,1,1}, {0,1,1,0}}
        \end{tikzpicture}
    \end{equation*}
and consider linear extension $(\{s_1, s_2\}, \{s_4\}, \{s_2,s_3\})$ of $\mcD$.  If $\lambda:\mcD\rightarrow W$ is the maximal labelling, then
$$\Lambda(\mcD)=\overline{\lambda}(\{s_2,s_3\})\overline{\lambda}(\{s_4\})\overline{\lambda}(\{s_1,s_2\})=(s_2s_3)(s_4)(s_1s_2s_1).$$
Corollary \ref{C:bpstaircase} implies that $(s_2s_3)(s_4s_1s_2s_1)$ is a BP decomposition with respect to $J_1:=\{s_1,s_2,s_4\}$ and $(s_4)(s_1s_2s_1)$ is a BP decomposition with respect to $J_2:=\{s_1,s_2\}$.  Theorem \ref{T:bpgeom} implies that the fibre bundle structure on $G/B$
$$\xymatrix{
P_{J_2}/B\, \ar@{^{(}->}[r] &\, P_{J_1}/B\, \ar@{->>}[d] \ar@{^{(}->}[r] &\, G/B \ar@{->>}[d] \\
 & P_{J_1}/P_{J_2} & G/P_{J_1}}$$
induces the following fibre bundle structure on $X(\Lambda(\mcD))$:
$$\xymatrix{
X(s_1s_2s_1)\, \ar@{^{(}->}[r] &\, X(s_4s_1s_2s_1)\, \ar@{->>}[d] \ar@{^{(}->}[r] &\, X(\Lambda(\mcD)) \ar@{->>}[d] \\
 & X^{J_2}(s_4) & X^{J_1}(s_2s_3)}$$
Note that the Schubert varieties  $X(s_1s_2s_1), X^{J_2}(s_4), X^{J_1}(s_2s_3)$ are all smooth, and in fact are sub-Grassmannians. Hence the Schubert variety $X(\Lambda(\mcD))$ is an iterated fibre bundle of Grassmannian flag varieties, and in particular is smooth.
\end{example}

Singular Schubert varieties do not always have Billey-Postnikov decompositions
as in the example above. As shown in \cite{Ry87} (type $A$) and \cite{RS14}
(all finite types), if $X(w)$ is rationally smooth then $w$ always has a BP
decomposition. In the next three sections we use this fact to make a connection
between labelled staircase diagrams $(\mcD, \lambda)$ and rationally smooth
Schubert varieties $X(\Lambda(\mcD, \lambda))$.

\section{Staircase diagrams and complete Billey-Postnikov decompositions}

We say that a BP decomposition $w = vu$, $v \in W^J$, $u \in W_J$, is a
\emph{Grassmannian BP decomposition} if $|J| = |S(w)|-1$. A \emph{complete BP
decomposition} of an element $w \in W$ is a factorization $w = v_n \cdots v_1$
such that $v_i (v_{i-1} \cdots v_1)$ is a Grassmannian BP decomposition for
every $i=2,\ldots,n$.  We say $w\in W$ is \emph{maximal} if $w=u_{S(w)}$, the unique maximal element in $W_{S(w)}$.
\begin{defn}\label{D:nearlymax}
    A non-maximal element $w \in W$ is \emph{nearly-maximal} if $w$ has a
    Grassmannian BP decomposition $w = vu$ where $S(u)\subset S(v)$.  A
    labelling $\lambda$ of a staircase diagram $\mcD$ is \emph{nearly-maximal} if
    $\lambda(B)$ is either maximal or nearly-maximal for all $B \in \mcD$.
\end{defn}
Note that the maximal labelling of a staircase diagram defined in Definition
\ref{D:maximal_labelling} is nearly-maximal.
\begin{lemma}\label{L:nearlymax}
    An element $w \in W$ is either maximal or nearly-maximal if and only if $w$
    has a complete BP decomposition $w = v_n \cdots v_1$ with $S(v_{i-1})
    \subset S(v_i)$ for all $i=2,\ldots,n$.
\end{lemma}
\begin{proof}
    If $w$ has a complete BP decomposition as stated, then $v_{n-1} \cdots v_1$
    must be the maximal element of $W_{S(v_{n-1})}$. The other direction is clear.
\end{proof}

\begin{prop}\label{P:para_diag_support}
    Let $\lambda:\mcD\arr W$ be an nearly-maximal labelling over the Dynkin-Coxeter graph of $(W,S)$, and
    suppose $\Lambda(\mcD) = vu$ is the parabolic decomposition with respect to
    a subset $J \subset S$. Then
    \begin{equation*}
        S(v)=\bigcup_{\substack{s\not\in J\\ B\succeq\min(\mcD_s)}} B.
    \end{equation*}
\end{prop}
Proposition \ref{P:para_diag_support} differs from Proposition \ref{P:labelparabolic}
in that $J$ is not required to be the support of a subdiagram. We use the
following lemma for the proof.
\begin{lemma}\label{L:paraextension}
    Let $w\in W^J$ and $s\in S\setminus D_L(w)$, and write $sw=vu$ where $v\in
    W^J$ and $u\in W_J$.  Then $\ell(v)\geq \ell(w)$ and
    \begin{equation*}
        S(v)=
        \begin{cases}
            S(w)& \text{if $s \in J$ and $s$ commutes with $S(v)$}\\
            S(w)\cup \{s\}& \text{otherwise}.
        \end{cases}
    \end{equation*}
\end{lemma}
\begin{proof}
    Since $s \not\in D_L(w)$, $w < sw$ and hence $w \leq v$. If $s\in S(w)$, then the lemma is proved since $S(w)
    \subset S(v)$.  Otherwise, if $s \not\in S(w)$, then the lemma follows from Lemma \ref{L:adjacentpara}.
\end{proof}
\begin{proof}[Proof of Proposition \ref{P:para_diag_support}]
    The set of $B \succeq \min(\mcD_s)$ for $s \not\in J$ forms an upper order
    ideal $\mcI$ in $\mcD$. The minimal blocks of $\mcI$ are precisely the
    blocks $\min(\mcD_s)$ for some $s \not\in J$. If $B$ is not minimal,
    then $B$ must cover some block $B'$ of $\mcI$, and since $B \cup B'$ is
    connected, there is an element $s \in B \setminus B'$ which is adjacent to
    some element $t \in B'$. Of course $B$ still covers $B'$ in the chain
    $\mcD_s \cup \mcD_t$, and since $\mcD_s$ is a saturated subset of this
    chain, $B$ must be the unique minimal block of $\mcD_s$. We conclude that
    for every $B \in \mcI$, there is an $s \in B \setminus J_R(B)$ such that
    either $s \not\in J$ or $s$ is adjacent to a block of $\mcI$ covered
    by $B$.

    Now take a linear extension  $B_1,\ldots,B_n$ of $\mcD$ such that $\mcI =
    \{B_k,B_{k+1},\ldots, B_n\}$. By the previous paragraph, for every $i=k,\ldots,n$
    we can find an element $s_i \in B_i \setminus J_R(B_i)$ such that either
    $s_i \not\in J$ or $s_i$ is adjacent to one of $B_k, \cdots, B_{i-1}$.
    Since $\lambda(B_i)$ is either maximal or nearly-maximal, it is not hard to see
    (using, i.e., the complete BP decomposition in Lemma \ref{L:nearlymax}) that
    there is an element $x_i = t^i_{p_i} t^i_{p_i-1} \cdots t^i_1 \leq \lambda(B_i)$
    such that $t^i_{1},\ldots,t^i_{p_i}$ is an enumeration of $B_i$, $t^i_1
    = s_i$, and $t^i_j$ is adjacent to at least one of $t^i_1,\ldots,t^i_{j-1}$
    for all $j \geq 2$. Since $x_i \leq \lambda(B_i)$ and $D_R(x_i) = \{s_i\}$, we
    have that $x_i \leq \overline{\lambda}(B_i)$, and consequently
    \begin{equation*}
        x := x_n x_{n-1} \cdots x_k \leq \overline{\lambda}(B_n) \cdots \overline{\lambda}(B_k)
            \leq \Lambda(\mcD).
    \end{equation*}
    If we write $x = v' u'$ with $v' \in W^J$, $u' \in W_J$, then $v' \leq v$.
    But Lemma \ref{L:paraextension} implies that
    \begin{equation*}
        S(v') = \bigcup_{B \in \mcI} B
    \end{equation*}
    as desired.
\end{proof}

\begin{defn}\label{D:bpset}
    If $w \in W$, let
    \begin{equation*}
        \bp(w) := \{s \in S\ |\ \text{$w$ has a BP decomposition with respect to $S\setminus\{s\}$ } \}.
    \end{equation*}
\end{defn}

\begin{thm}\label{T:bpset}
    Let $\lambda$ be a nearly-maximal labelling of a staircase diagram $\mcD$. Then
    \begin{equation*}
        \bp(\Lambda(\mcD)) = \bigcup_{B \in \max(\mcD)} \bp(\lambda(B)) \setminus J_R(B),
    \end{equation*}
    where $\max(\mcD)$ is the set of maximal blocks of $\mcD$.
\end{thm}
\begin{proof}
    Suppose $s \in \bp(\Lambda(\mcD))$, and let $B = \min(\mcD_s)$. Let $\mcI$ be
    the upper order ideal of blocks $B' \succeq B$, and let $\Lambda(\mcD) = vu$ be the
    BP decomposition with respect to $J = \supp(\mcD) \setminus \{s\}$.
    By Proposition \ref{P:para_diag_support}, $S(v) = \supp(\mcI)$. Now
    $\flip(\mcI)$ is a lower order ideal in $\flip(\mcD)$, so by Theorem
    \ref{T:labelling} and Proposition \ref{P:labelparabolic}, $\Lambda(\mcD) =
    \Lambda(\mcI) v'$, $v' \in {}^{\supp(\mcI)} W$, is the left parabolic
    decomposition with respect to $\supp(\mcI)$, where $\Lambda(\mcI)$ comes
    from the restriction of $\lambda$ to $\mcI$. Since $\Lambda(\mcD) = vu$ is a BP
    decomposition, we have that $u = u_0 u'$, where $u_0$ is the maximal element of
    $W_{S(v) \cap J}$ and $u' \in {}^{S(v) \cap J} W_J \subset {}^{S(v)} W$.
    We conclude that $\Lambda(\mcI) = v u_0$ is either maximal or nearly-maximal.

    Now we claim that
    $\mcI = \{B\}$. By construction, $B$ is the unique minimal block of
    $\mcI$. Indeed, suppose there is some other block $B' \in \mcI$. As in the proof
    of Proposition \ref{P:para_diag_support}, there must be some $t \in B'
    \setminus J_R(B')$ which is adjacent to some $B''$ covered by $B'$.
    By Theorem \ref{T:labelling}, $t \not\in D_R(\Lambda(\mcI))$, so $\Lambda(\mcI)$
    must be nearly-maximal rather than maximal.  Hence $t=s$ is the unique element
    of $\supp(\mcI) \setminus D_R(\Lambda(\mcI))$. Since every element of $\mcI$
    greater than $B$ and $B'$ will decrease the size of the descent set, we must have $\mcI
    = \{B, B'\}$ and $\Lambda(\mcI) = \overline{\lambda}(B') \lambda(B)$, where
    $\overline{\lambda}(B') := \lambda(B') u_{J_R(B', \mcI)}$. By Corollary
    \ref{C:bpstaircase}, $\Lambda(\mcI) = \overline{\lambda}(B') \lambda(B)$ is
    the BP decomposition with respect to $J$, and consequently $v =
    \overline{\lambda}(B')$. But then $S(v) = B'$ contains $\supp(\mcI)$, a
    contradiction, so we must have $\mcI = \{B\}$ as claimed.

    So far we've shown that if $s \in \bp(\Lambda(\mcD))$, then $B = \min(\mcD_s)$ must be
    maximal in $\mcD$. The argument above also shows that $s$ belongs to
    $\bp(\lambda(B))$ where $\lambda(B) = v u_0$. For the converse, suppose $s
    \in B \setminus J_R(B)$ for some $B \in \max(\mcD)$. Once again, we know
    from applying Theorem \ref{T:labelling} and Proposition
    \ref{P:labelparabolic} that $\Lambda(\mcD) = \lambda(B) v'$, where $v' \in
    {}^{\supp(\mcI)} W$. If $s \in \bp(\lambda(B))$, then $\lambda(B)$ has a BP
    decomposition $\lambda(B) = x y$ where $x \in W^{B \setminus \{s\}}$ and $y
    \in W_{B \setminus \{s\}}$. Then $\Lambda(\mcD) = x (y v')$ is a BP
    decomposition with respect to $J$, proving the theorem.
\end{proof}
\begin{rmk}
    If $\lambda$ is a nearly-maximal labelling of $\mcD$, and $B \in \max(\mcD)$, then
    $\bp(\lambda(B)) \setminus J_R(B)$ is non-empty. Indeed, if $\lambda(B)$ is maximal then
    $$\bp(\lambda(B)) \setminus J_R(B) = B \setminus J_R(B)$$ is non-empty. If
    $\lambda(B)$ is nearly-maximal, then $$\bp(\lambda(B)) = B \setminus D_R(\lambda(B))$$ contains
    exactly one element, and is contained in $B \setminus J_R(\lambda(B))$.
\end{rmk}

\section{The bijection theorem}

We can now state and prove the main structural theorem of this paper:
\begin{thm}\label{T:staircasebij}
    Let $W$ be a Coxeter group. Then the map $\phi : (\mcD, \lambda) \mapsto \Lambda(\mcD)$
    defines a bijection between staircase diagrams $\mcD$ with a nearly-maximal
    labelling $\lambda$, and elements of $W$ with a complete BP decomposition.
\end{thm}

\begin{lemma}\label{L:oneblockmore}
    Suppose $\lambda$ is a nearly-maximal labelling of a staircase diagram $\mcD$ and
    $v$ is a Grassmannian element of $W^{\supp(\mcD)}$ such that $v \Lambda(\mcD)$ is
    a BP decomposition. Then
    \begin{enumerate}[(a)]
        \item $\mcD^0 := \{B \in \mcD\ |\ B \not\subseteq S(v)\}$ is a lower
            order ideal in $\mcD$,
        \item $\widetilde{\mcD} := \mcD^0 \cup \{S(v)\}$ is a staircase
            diagram with the additional covering relations $\max(\mcD^0_s)\prec S(v)$ for every $s \in \supp(\mcD^0)$ contained in or adjacent to $S(v)$ and
        \item the function $\widetilde{\lambda} : \widetilde{\mcD} \arr W$ defined by
            \begin{equation*}
                \widetilde{\lambda}(B) := \begin{cases} \lambda(B) & \text{if $B \in \mcD_0$} \\
                                                 v u_{S(v) \cap \supp(\mcD)} & \text{if $B = S(v)$}
                                    \end{cases}
            \end{equation*}
            is a nearly-maximal labelling of $\widetilde{\mcD}$, with
            $\Lambda(\widetilde{\mcD}, \widetilde{\lambda}) = v \Lambda(\mcD)$.
    \end{enumerate}
\end{lemma}
\begin{proof}
    Let $B\in\mcD$.  Since $B$ is connected, if $J_L(B)$ is non-empty then there must be some $s
    \in B \setminus J_L(B)$ which is adjacent to some $t \in J_L(B)$. Since
    $\max \mcD_s = B \preceq \max \mcD_t$, the element $s \not\in D_L(\mcD)$.
    Similarly, if $J_L(B) = \emptyset$ but $B$ is covered by a block $B' \in
    \mcD$, then there is some $s \in B \setminus B'$ which is adjacent to
    some $t \in B'$, and again $s \not\in D_L(\mcD)$. We conclude that if $B
    \subseteq D_L(\mcD)$ then $B$ must be a maximal element of $\mcD$.
    Thus if $B \subset S(v)$, then $B \subset D_L(\Lambda(\mcD)) \subset D_L(\mcD)$,
    and $\mcD^0$ is a lower order ideal of $\mcD$.  This proves part (a).

    \smallskip

    For part (b), since $\mcD^0$ is a subdiagram of $\mcD$, we only need
    to check the conditions of Definition \ref{D:staircase} on $S(v)$. If $S(v)$ covers $B \in \mcD^0$ then
    $S(v) \cup B$ is connected by construction. Similarly if $s \in S(v)
    \cap \supp(\mcD^0)$ then $S(v) \succeq \max(\mcD^0_s)$, and if $s \adj t$
    then $S(v) \succeq \max(\mcD^0_t)$, so $\widetilde{\mcD}_s$ and
    $\widetilde{\mcD}_s \cup \widetilde{\mcD}_t$ are both chains. By
    Corollary \ref{C:bpstaircase}, the element $s\in D_L(\Lambda(\mcD^0))$, so $\max(\mcD^0_s \cup \mcD^0_t) = \max(\mcD^0_s)$, and
    consequently $\widetilde{\mcD}_s$ is a saturated subset of
    $\widetilde{\mcD}_s \cup \widetilde{\mcD}_t$. The chain
    $\widetilde{\mcD}_t$ is also a saturated subset of $\widetilde{\mcD}_s \cup
    \widetilde{\mcD}_t$, since if $t \not\in S(v)$ then $\widetilde{\mcD}_t =
    \mcD^0_t$. Finally, if $S(v) \setminus \supp(\mcD) = \{s_0\}$ then
    $S(v)$ is both the maximal and minimal block of $\widetilde{\mcD}_{s_0}$.
    If $B \in \mcD^0$ then $J_R(B; \widetilde{\mcD}) = J_R(B; \mcD) \subsetneq
    B$, while $J_L(B; \widetilde{\mcD}) = J_L(B; \mcD) \cup (B \cap S(v))$. If
    $J_L(B; \mcD) = \emptyset$, then $J_L(B; \widetilde{\mcD}) \subsetneq B$ by
    the definition of $\mcD_0$. If $J_L(B; \mcD)$ is non-empty, then as in the
    first paragraph we can find an element $t$ of $B \setminus J_L(B; \mcD)$
    which is not in $D_L(\mcD)$, and hence $t$ will be contained in $B
    \setminus J_L(B; \widetilde{\mcD})$. We conclude that $\widetilde{\mcD}$
    is a staircase diagram, proving part (b).

    \smallskip

    For part (c), we observe that $$J_R(S(v); \widetilde{\mcD}) = S(v) \cap
    \supp(\mcD^0) \subseteq S(v) \cap \supp(\mcD),$$ while $J_L(S(v);
    \widetilde{\mcD}) = \emptyset$.  Hence $J_R(S(v)) \subset D_R(\widetilde{\lambda}(S(v)))$
    and $J_L(S(v)) \subset D_L(\widetilde{\lambda}(S(v)))$.  Similarly,
    $$S(\widetilde{\lambda}(S(v)) u_{J_R(S(v))}) = S(u_{J_L(S(v))}
    \widetilde{\lambda}(S(v))) = S(v),$$ and since $v$ is Grassmannian,
    $\widetilde{\lambda}(S(v))$ is nearly-maximal by construction.

    Now if $B \in \mcD^0$, then as stated before, $J_R(B; \widetilde{\mcD}) = J_R(B;
    \mcD)$ which implies $J_R(B; \widetilde{\mcD}) \subseteq D_R(\widetilde{\lambda}(B))$ and
    $S(\widetilde{\lambda}(B) u_{J_R(B)}) = B$. Suppose $s \in J_R(B;
    \widetilde{\mcD}) \setminus J_R(B; \mcD)$. This occurs if and only if $B =
    \max \mcD_s$ and $s \in S(v)$. The latter condition implies that $s\in D_L(\Lambda(\mcD))$, which in turn implies that $s \in D_L(\lambda(B))$.
    We conclude that $J_L(B; \widetilde{\mcD}) \subseteq D_L(\lambda(B))$. It remains
    to show that $S(u_{J_L(B; \widetilde{\mcD})} \lambda(B)) = B$. This clearly
    occurs if $\lambda(B)$ is maximal, so suppose $\lambda(B)$ is nearly-maximal. Let $$K = D_R(\lambda(B)) = B \setminus \{s_0\},$$ and decompose $\lambda(B) = x u_{K}$, where $x \in W^{K}$. Because $x$ is Grassmannian and $S(x) = B$,
    we can pick an element $s_k \in B \setminus J_R(B; \widetilde{\mcD})$
    and then find a sequence $s_{k}, \ldots, s_{1}, s_0$ in $B$ such that
    $s_{i}$ is adjacent to $s_{i-1}$ for all $i$, and $x' = s_{k} \cdots s_0
    \leq x$. Since $x' \in W^{B \setminus \{s_k\}}$ and $s_0$ is adjacent to
    each connected component of $K$, we can find (using, i.e., Lemma
    \ref{L:paraextension}) an element $x'' \in W^{B \setminus \{s_k\}}$ with
    $x'' \leq x u_{K}$ and $S(x'') = B$. This implies that $S(u_{J_L(B;
    \widetilde{\mcD})} \lambda(B)) = B$ as desired, and we conclude that $\widetilde{\lambda}$
    is a nearly-maximal labelling of $\widetilde{\mcD}$.

    To finish the proof, we revisit the blocks $B \in \mcD \setminus \mcD_0$.
    All such blocks are contained in $D_L(\Lambda(\mcD))$ and are maximal in $\mcD$.
    As a result, $\lambda(B) = u_B$ for all $B \in \mcD \setminus \mcD_0$, and
    furthermore all these elements commute. Also, all the elements of
    $\mcD \setminus \mcD_0$ appear among the connected components $J_1,\ldots,
    J_k$ of $S(v) \cap \supp(\mcD)$. If a component $J_i$ does not belong to
    $\mcD \setminus \mcD^0$, then necessarily $J_i \subset \supp(\mcD^0)$.
    We conclude that
    \begin{align*}
        v \Lambda(\mcD) & = v \left( \prod_{B \in \mcD \setminus \mcD_0} u_B u_{J_R(B; \mcD)} \right)
           \Lambda(\mcD^0)
         =v \left( \prod_{B \in \mcD \setminus \mcD_0} u_B u_{B \cap \supp(\mcD^0)} \right) \Lambda(\mcD^0) \\
        & =v \left( \prod_{i=1}^k u_{J_i} u_{J_i \cap \supp(\mcD^0)} \right) \Lambda(\mcD^0)
         = v \cdot u_{S(v) \cap \supp(\mcD)}\cdot  u_{S(v) \cap \supp(\mcD^0)} \cdot \Lambda(\mcD^0) \\
        & = \widetilde{\lambda}(S(v))\cdot  u_{S(v) \cap \supp(\mcD^0)}\cdot  \Lambda(\mcD^0) = \Lambda(\widetilde{\mcD}).
    \end{align*}
\end{proof}
Using Lemma \ref{L:oneblockmore} and the results of the previous sections,
we can now prove the main theorem.
\begin{proof}[Proof of Theorem \ref{T:staircasebij}]
    By Corollary \ref{C:bpstaircase} and Lemma \ref{L:nearlymax}, if $\lambda$
    is a nearly-maximal labelling of $\mcD$ then $\Lambda(\mcD)$ has a complete
    BP decomposition.

    Now let $\lambda_i$ be a nearly-maximal labelling of a staircase diagram $\mcD_i$,
    where $i=1,2$ such that $\Lambda(\mcD_1) = \Lambda(\mcD_2)$. Choose some element
    $s \in \bp(\Lambda(\mcD_1))$, and let $\Lambda(\mcD_1) = vu$ be the BP decomposition with respect to $J = S \setminus \{s\}$. As in the proof
    of Theorem \ref{T:bpset}, $\Lambda(\mcD_1)$ must have left-parabolic
    decomposition $\Lambda(\mcD_1) = \lambda_1(B) v'$ with respect to $B = S(v)$, where $B$
    is the unique maximal block of $\mcD$ containing $s$. Let $\mcD_1' =
    \mcD_1 \setminus \{B\}$. By Proposition \ref{P:labelparabolic}, $S(v') =
    \supp(\mcD_1')$. But $s\in \bp(\Lambda(\mcD_2))$, so $B$ is a
    maximal block of $\mcD_2$, and $\lambda_1(B) = \lambda_2(B)$. If $\mcD_2' = \mcD_2
    \setminus \{B\}$, then $J_R(B, \mcD_i) = B \cap \supp(\mcD_i') = B \cap
    S(\lambda_i(B) \Lambda(\mcD_i))$ is independent of $i$, so $\overline{\lambda}_1(B)
    = \overline{\lambda}_2(B)$ and hence $\Lambda(\mcD_1') = \Lambda(\mcD_2')$. By induction on $|\mcD|$, the map $\phi$ must be injective.

    Finally we use induction on $|S(w)|$ to show that $\phi$ is
    surjective. Suppose $x \in W$ has a complete BP decomposition $x = v_n
    \cdots v_1$, and let $(\mcD, \lambda)$ be the nearly-maximal labelled staircase
    diagram such that $\Lambda(\mcD) = v_{n-1} \cdots v_1$. Then by Lemma \ref{L:oneblockmore},
    there is a staircase diagram $\widetilde{\mcD}$ with a nearly-maximal
    labelling $\widetilde{\lambda}$ such that $\Lambda(\widetilde{\mcD}) = x$.
\end{proof}

\section{Staircase diagrams and rationally smooth elements}

In this section we combine Theorem \ref{T:staircasebij} with the previously
mentioned existence theorems for Billey-Postnikov decompositions to get
bijections between certain labelled staircase diagrams and rationally smooth
elements.  We say $w\in W$ is \emph{rationally smooth} if the Bruhat interval
$[e,w]$ is rank symmetric with respect to length. By the Carrell-Peterson
theorem, this condition is equivalent to the corresponding Schubert variety
$X(w)$ being rationally smooth \cite{Ca94}.

\begin{defn}
    A labelling $\lambda:\mcD\arr W$ is \emph{rationally smooth}
    if and only if $\lambda(B)$ is rationally smooth for all $B \in \mcD$.
\end{defn}

\begin{thm}\label{T:ratsmoothbij}
    Let $\lambda:\mcD\arr W$ be a labelling of a staircase diagram $\mcD$. Then $\Lambda(\mcD)$
    is rationally smooth if and only if $\lambda:\mcD\arr W$ is rationally smooth.
\end{thm}
\begin{proof}
    By Corollary \ref{C:bpstaircase}, $\Lambda(\mcD) = \overline{\lambda}(B) \Lambda(\mcD_0)$
    is a BP decomposition, where $B$ is some maximal block of $\mcD$ and
    $\mcD_0 = \mcD \setminus \{B\}$. By \cite[Theorem 3.3 and Remark 4.5]{RS14}, $\Lambda(\mcD)$ is
    rationally smooth if and only if $\Lambda(\mcD_0)$ and $\lambda(B) = \overline{\lambda}(B)
    u_{J_R(B)}$ are rationally smooth, so the theorem follows by induction.
\end{proof}

If $W$ is a finite Weyl group, then every rationally smooth element has a
complete BP decomposition \cite[Theorem 3.6]{RS14}, so we get the following:
\begin{cor}\label{C:ratsmoothbij}
    If $W$ is a finite Weyl group, then there is a bijection between rationally
    smooth elements of $W$, and staircase diagrams over $W$ with a rationally
    smooth nearly-maximal labelling.
\end{cor}
Similar results hold for the affine Weyl group of type $\widetilde{A}$ \cite{BC12}
and right-angled and long-braid Coxeter groups \cite{RS12}.


If $w = vu_J$ with $J=D_R(w)$ is nearly-maximal and rationally smooth, then the
associated Grassmannian elements $v \in W^J$ are completely listed in
\cite[Theorem 3.8]{RS14}, and hence Corollary \ref{C:ratsmoothbij} is quite
concrete.  In particular, if $W$ is simply-laced, then rationally smoothness is
equivalent to smoothness, and the maximal elements are the only rationally
smooth nearly-maximal elements.  Thus:
\begin{cor}\label{C:simplylacedbij}
    If $W$ is a simply-laced finite type Weyl group, then there is a bijection
    between staircase diagrams over the Dynkin diagram of $W$ and smooth elements of $W$.
\end{cor}
For non-simply-laced types, all nearly-maximal, rationally smooth elements
are almost-maximal in the following sense \cite[Corollary 5.10]{RS14}:
\begin{defn}\label{D:almostmaxlabelling}
    An element $w \in W$ is almost-maximal if both $w$ and $w^{-1}$ are
    nearly-maximal.
\end{defn}
Thus we can replace nearly-maximal with almost-maximal in Corollary
\ref{C:ratsmoothbij}. In any Coxeter group, $w \in W$ is rationally smooth
if and only if $w^{-1}$ is rationally smooth. If $w = \Lambda(\mcD)$, then $w^{-1}
= \Lambda(\flip(\mcD))$ by Theorem \ref{T:labelling}, so $\flip(\mcD)$ is the
labelled staircase diagram associated to $w^{-1}$.

\begin{rmk}
    In \cite{RS14}, an element $w \in W^J$ is said to be almost-maximal
    relative to $J$ if $w u_J$ is almost-maximal in the above sense. We
    can say that an element $x$ has a complete almost-maximal BP decomposition
    if $x = v_n \cdots v_1$, where $v_i (v_{i-1} \cdots v_1)$ is a BP decomposition
    and $v_i$ is either maximal or almost-maximal relative to $J_i = S(v_{i-1}
    \cdots v_1)$ for all $i$. Then Theorems \ref{T:staircasebij} and \ref{T:bpset}
    give a bijection between almost-maximally labelled staircase diagrams and
    elements with a complete almost-maximal BP decomposition, holding for any
    Coxeter group. One interesting consequence is that $w$ has a complete
    almost-maximal BP decomposition if and only if $w^{-1}$ has a complete
    almost-maximal BP decomposition.
\end{rmk}


\section{Staircase diagrams and Catalan numbers}\label{S:staircasecatalan}


The rest of this paper is concerned with enumerating staircase diagrams (and
labelled staircase diagrams) for the classical finite-type Coxeter groups. Before
going through each case individually, we look at what the classical types have
in common.  We begin with a few important definitions on staircase diagrams.

\begin{definition}
    We say $s\in S$ is \emph{critical point} of $\mcD$ if $|\mcD_s|=1$.  The collection of critical points is called the critical set of $\mcD$.
\end{definition}
\begin{example}
    In Example \ref{Ex:typeA_11}, the critical set of $\mcD$ is $\{s_1,s_6,s_8,s_9,s_{11}\}$.
\end{example}

The following lemma is immediate from part (4) of Definition \ref{D:staircase}.
\begin{lemma}\label{L:extreme_crits}
    If $B\in \mcD$ is an extremal block, then $B$ contains a critical point of $\mcD$.
\end{lemma}
\begin{definition}
    We say $\mcD$ is \emph{elementary diagram} if the critical set of $\mcD$ is contained in the leaves of the support $\supp(\mcD)$.
\end{definition}

The fundamental principle we use to enumerate staircase diagrams is to first
decompose a diagram into elementary diagrams along critical points.  For
example, we have

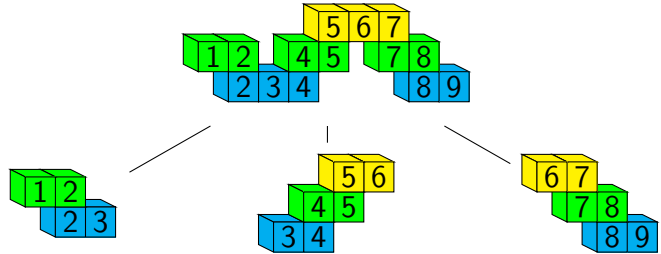
\begin{equation*}
\begin{tikzpicture}[level distance=20mm,
level 1/.style={sibling distance=35mm}]
    \node {\begin{tikzpicture}[scale=0.4]
            \planepartition{9}{
                {1,1,0,0,0,1,1,1,0},
                {0,1,1,0,1,1,0,1,1},
                {0,0,1,1,1}}
            \end{tikzpicture}}
        child{ node{
        \begin{tikzpicture}[scale=0.4]
            \planepartition{3}{
                {1,1,0},
                {0,1,1}}
        \end{tikzpicture}}}
        child{ node{
        \begin{tikzpicture}[scale=0.4]
            \planepartition{7}{
                {0,0,0,1,1},
                {0,0,1,1,0},
                {0,1,1,0}}
        \end{tikzpicture}}}
        child{ node {
        \begin{tikzpicture}[scale=0.4]
            \planepartition{9}{
                {1,1,0,0},
                {0,1,1,0},
                {0,0,1,1}}
        \end{tikzpicture}}};
\end{tikzpicture}
\end{equation*}

For this reason, we focus on elementary diagrams.  We begin by considering a
particular family of graphs and the relationship between their elementary
diagrams. Suppose we have a fixed graph $\Gamma$ with distinguished vertex $s\in S$, and vertex
set $S$ of size $q$. Define the graph $\Gamma_{q+p}$ to be the graph where we
attach a line graph of $p$ vertices to the vertex $s$ and let $S_n$ denote the
set of vertices in $\Gamma_n$.  In particular, $\Gamma_q=\Gamma$ and for any
$n\geq q$, $\Gamma_n$ is a graph with $n$ vertices.  Set $s_q = s$, and for $n
> q$ let $s_n$ denote the new leaf in the graph $\Gamma_n$. Since we often work
with trees, it is convenient to make the following definition:
\begin{defn}
    If $\Gamma$ is a tree, and $t,r \in S$, then $[t,r] \subseteq S$ will denote
    the vertices of the unique path connecting $t$ and $r$, with endpoints
    included.
\end{defn}
Next, define $Z_\Gamma(n)$ to be the set of fully supported elementary diagrams
on the graph $\Gamma_n$. Define
\begin{equation*}
    Z_{\Gamma}^+(n):=\{\mcD\in Z_\Gamma(n)\ |\ \text{$s_n$ is contained in a maximal block}\}.
\end{equation*}
Note that the maximal block containing $s_n$ is unique since $\mcD_{s_n}$ is a
chain.  If $\mcD\in Z_{\Gamma}^+(n)$, let $B_{\mcD}:=\max(\mcD_{s_n})$.  The
following is an algorithm for constructing elementary diagrams in
$Z_\Gamma^+(n+1)$ from elementary diagrams in $Z_\Gamma^+(n)$.

\begin{definition}\label{D:Cat_gen}
    Let $\mcD\in Z_\Gamma^+(n)$ and let $P(B_{\mcD})$ denote the collection of connected, proper, and nonempty subsets of $B_{\mcD}$ containing $s_n$.  Define the set of staircase diagrams $\mathfrak{G}_p(\mcD)\subseteq Z_\Gamma^+(n+p)$ as follows:

    First, if $s_n$ is not a critical point of $\mcD$, then let $\mcD\cup\{s_{n+1}\}$, with the additional covering relation $B_{\mcD}\prec\{s_{n+1}\}$, be the single elementary diagram in the set $\mathfrak{G}_1(\mcD)$.

    Otherwise, if $s_n$ is a critical point of $\mcD$, then define
    \begin{equation*}
        \mathfrak{G}_1(\mcD):=\{\mcD^0\}\cup \bigcup_{B'\in P(B_{\mcD})}\{\mcD^{B'}\}
    \end{equation*}
        where $\mcD^0:=\{B^0\ |\ B\in\mcD\}$ with the same covering relations as $\mcD$ and
    $$B^0:=
        \begin{cases}
            B & \text{if}\ B\notin \mcD_{s_{n-1}}\cup\{B_{\mcD}\}\\
            B\cup\{s_n\} & \text{if}\  B\in \mcD_{s_{n-1}}\setminus\{B_{\mcD}\}\\
            B\cup\{s_{n+1}\}& \text{if}\ B=B_{\mcD}
        \end{cases},$$
    and for $B'\in P(B_{\mcD})$ we let
    \begin{equation*}
        \mcD^{B'}:=\mcD\cup\{B'\cup\{s_{n+1}\}\}
    \end{equation*}
    with the additional covering relation $B_{\mcD}\prec B'\cup\{s_{n+1}\}$.  Recursively define
    $$\mathfrak{G}_{p+1}(\mcD):=\bigcup_{\mcG\in \mathfrak{G}_{p}(\mcD)}\ \mathfrak{G}_{1}(\mcG).$$
\end{definition}

\begin{example}\label{Ex:Catalan_genD}
    Let $\Gamma$ be the Dynkin graph of type $D_4$ with $S=\{s_1,s_2,s_3,s_4\}$ and fix $s=s_4.$  Then $\Gamma_n$ is the Dynkin graph of type $D_n$.
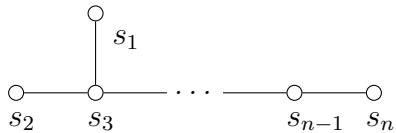
    \begin{equation*}
    \begin{tikzpicture}
        \begin{scope}[start chain]
            \dnode{2} \dnode{3} \node[chj,draw=none] {\dots}; \dnode{n-1} \dnode{n}
        \end{scope}
        \begin{scope}[start chain=br going above]
            \chainin(chain-2); \dnodebr{1}
        \end{scope}
    \end{tikzpicture}
\end{equation*}
    Let $\mcD=([s_1,s_4]\prec[s_2,s_4])\in Z^+_\Gamma(4)$, then $\mathfrak{G}_p(\mcD)$ for $p=1,2,3,4$ are given by the following diagrams:

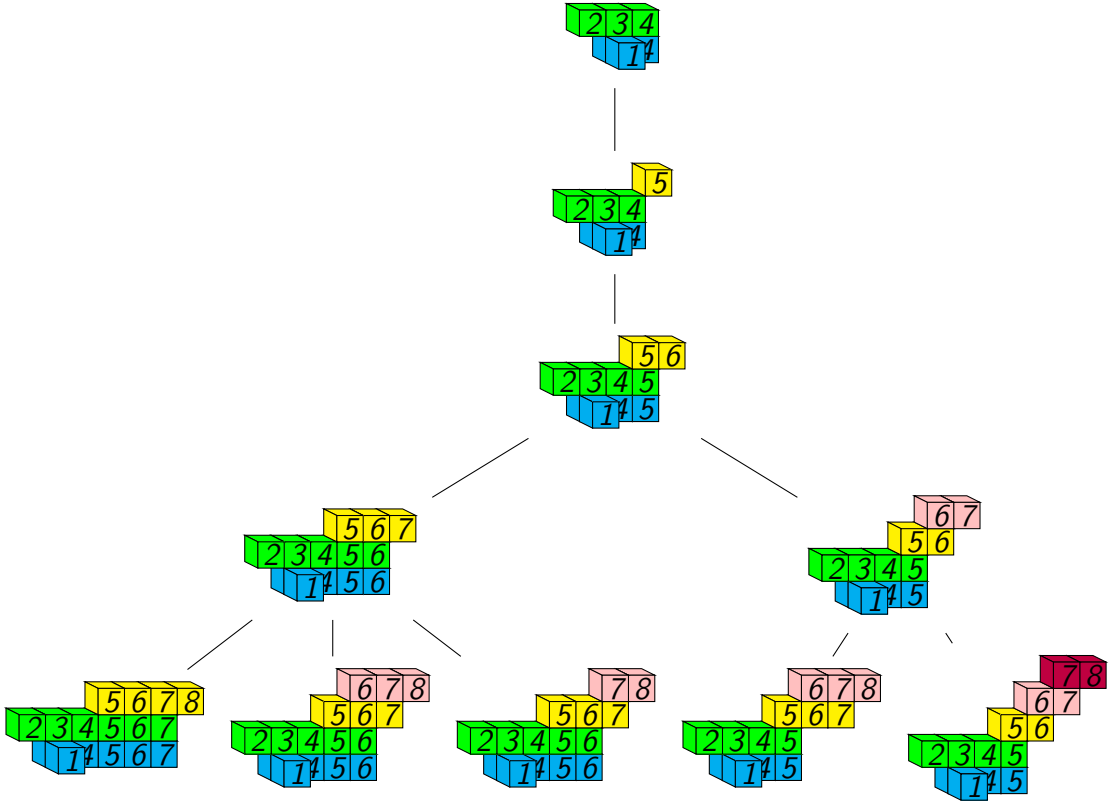
\begin{tikzpicture}[level distance=23mm,
level 1/.style={sibling distance=10mm},
level 2/.style={sibling distance=10mm},
level 3/.style={sibling distance=75mm},
level 4/.style={sibling distance=30mm}]
    \node {
        \begin{tikzpicture}[scale=0.35]
            \planepartitionD{4}{
                {1,2,0},
                {1,1,1}}
        \end{tikzpicture}}
        child { node {
            \begin{tikzpicture}[scale=0.35]
            \planepartitionD{5}{
                {0,1,2,0},
                {0,1,1,1},
                {1,0,0,0}}
            \end{tikzpicture}}
            child { node {
                \begin{tikzpicture}[scale=0.35]
                \planepartitionD{6}{
                    {0,1,1,2,0},
                    {0,1,1,1,1},
                    {1,1,0,0,0}}
                \end{tikzpicture}}
                child { node {
                    \begin{tikzpicture}[scale=0.35]
                    \planepartitionD{7}{
                        {0,1,1,1,2,0},
                        {0,1,1,1,1,1},
                        {1,1,1,0,0,0}}
                    \end{tikzpicture}}
                    child { node {
                        \begin{tikzpicture}[scale=0.35]
                        \planepartitionD{8}{
                            {0,1,1,1,1,2,0},
                            {0,1,1,1,1,1,1},
                            {1,1,1,1,0,0,0}}
                        \end{tikzpicture}}}
                    child { node {
                        \begin{tikzpicture}[scale=0.35]
                        \planepartitionD{8}{
                            {0,0,1,1,1,2,0},
                            {0,0,1,1,1,1,1},
                            {0,1,1,1,0,0,0},
                            {1,1,1,0,0,0,0}}
                        \end{tikzpicture}}}
                    child { node {
                        \begin{tikzpicture}[scale=0.35]
                        \planepartitionD{8}{
                            {0,0,1,1,1,2,0},
                            {0,0,1,1,1,1,1},
                            {0,1,1,1,0,0,0},
                            {1,1,0,0,0,0,0}}
                        \end{tikzpicture}}}}
                child { node {
                    \begin{tikzpicture}[scale=0.35]
                    \planepartitionD{7}{
                        {0,0,1,1,2,0},
                        {0,0,1,1,1,1},
                        {0,1,1,0,0,0},
                        {1,1,0,0,0,0}}
                    \end{tikzpicture}}
                    child { node {
                        \begin{tikzpicture}[scale=0.35]
                        \planepartitionD{8}{
                            {0,0,0,1,1,2,0},
                            {0,0,0,1,1,1,1},
                            {0,1,1,1,0,0,0},
                            {1,1,1,0,0,0,0}}
                        \end{tikzpicture}}}
                    child { node {
                        \begin{tikzpicture}[scale=0.35]
                        \planepartitionD{8}{
                            {0,0,0,1,1,2,0},
                            {0,0,0,1,1,1,1},
                            {0,0,1,1,0,0,0},
                            {0,1,1,0,0,0,0},
                            {1,1,0,0,0,0,0}}
                        \end{tikzpicture}}}}}};
\end{tikzpicture}
\end{example}

It is easy to see that if $\mcD\in Z_{\Gamma}^+(n)$, then $\mathfrak{G}_p(\mcD)\subseteq Z_{\Gamma}^+(n+p)$.


\begin{lemma}\label{L:Catalan_disjoint}
If $\mcD,\mcG\in Z^+_\Gamma(n)$ and $\mcD\neq \mcG$, then $\mathfrak{G}_p(\mcD)\cap\mathfrak{G}_p(\mcG)=\emptyset$ for all $p>0$.
\end{lemma}

\begin{proof}
It suffices to show the sets $\mathfrak{G}_1(\mcD)$ and $\mathfrak{G}_1(\mcG)$ are disjoint if $\mcD\neq \mcG$.  First, if $s_n$ is not a critical point of either $\mcD$ or $\mcG$, then the lemma is true.  If $B_{\mcD}=\{s_n\}$, then again the lemma is true.  Assume that $|B_{\mcD}|\geq 2$ and suppose $\mcD^0=\mcG^{B'}$ for some $B'\in P(B_{\mcD})$.  If $|\mcD_{s_{n-1}}|=2$, then $\mcG$ is not elementary since $\mcG_{s_{n-1}}=1$.  If  $|\mcD_{s_{n-1}}|>2$, then $|\mcD^0_{s_n}\setminus\mcD^0_{s_{n+1}}|\geq 2$.  But then $s_n$ is not a critical point of $\mcG$.  Hence the sets $\mathfrak{G}_1(\mcD)$ and $\mathfrak{G}_1(\mcG)$ are disjoint.
\end{proof}

Lemma \ref{L:Catalan_disjoint} implies that the size of the set $\mathfrak{G}_p(\mcD)$ grows predictably.  Let $\cc_n:=\frac{1}{n+1}\binom{2n}{n}$ denote the $n$-th Catalan number.

\begin{proposition}\label{P:Catalan_gen}
    Let $\mcD\in Z_\Gamma^+(n)$ with $n\geq 3$.  If $s_n$ is not a critical point of $\mcD,$ then
    \begin{enumerate}
        \item[(1)] $|\mathfrak{G}_p(\mcD)|=\cc_{p-1}$.
    \end{enumerate}
    Otherwise, let $\mcG,\mcD,\mcH\in Z_D^+(n)$ such that $|B_{\mcG}|=1,|B_{\mcD}|=2$ and $|B_{\mcH}|=3$.   If $s_n$ is a critical point of $\mcG,\mcD,\mcH$ then the following are true:
    \begin{enumerate}
    \item[(2)] $|\mathfrak{G}_p(\mcG)|=\cc_{p}.$
    \item[(3)] $|\mathfrak{G}_p(\mcD)|=\cc_{p+1}.$
    \item[(4)] $|\mathfrak{G}_p(\mcD)\cup\mathfrak{G}_p(\mcH)|=\cc_{p+2}$
    \end{enumerate}
\end{proposition}

\begin{proof}Parts (1) and (2) follow directly from part (3).  Note that $B_{\mcD}=\{s_{n-1},s_n\}$ and hence if $\mcD'\in\mathfrak{G}_p(\mcD),$ then $B_{\mcD'}$ is a connected interval contained in $[s_n,s_{n+p}]$.  Let $\cc_{p,k}$ denote the number of staircase diagrams $\mcD'\in \mathfrak{G}_p(\mcD)$ for which $B_{\mcD'}$ is an interval of size $k$.
It is easy see from Definition \ref{D:Cat_gen} that $\cc_{p,p+2}=1$ for all $p\geq 0$ and that $\cc_{p,p+k}=0$ for all $k\geq 3$.  Lemma \ref{L:Catalan_disjoint} implies the recursion
    \begin{equation}\label{Eq:Catalan_triangle}
        \cc_{p+1,k}=\sum_{i={k-1}}^{p+2} \cc_{p,i}.
    \end{equation}
    Equation \eqref{Eq:Catalan_triangle} is satisfied by Catalan's triangle and hence
    \begin{equation*}
        |\mathfrak{G}_p(\mcD)|=\sum_{k=2}^{p+2} \cc_{p,k}=\cc_{p+1}.
    \end{equation*}
    To prove part (4), note that $\mathfrak{G}_1(\mcD)$ contains exactly two elementary diagrams $\mcD',\mcD''$ where $|B_{\mcD'}|=2$ and $|B_{\mcD''}|=3$.  Hence we can identify $\mcD'\mapsto\mcD$ and $\mcD''\mapsto\mcH$.  This induces a bijection between the sets $\mathfrak{G}_p(\mcD)\cup\mathfrak{G}_p(\mcH)$ and $\mathfrak{G}_{p+1}(\mcD)$.  Part (4) now follows from part (3).
\end{proof}

In Example \ref{Ex:Catalan_genD}, we have that $|\mathfrak{G}_3(\mcD)|=\cc_2=2$ and $|\mathfrak{G}_4(\mcD)|=\cc_3=5$.  We now describe the image of $\mathfrak{G}_1$ when applied to all elements of $Z_\Gamma^+(n).$

\begin{lemma}\label{L:Catalan_almostsurjective}
    The set
    $$\bigcup_{\mcD'\in Z_\Gamma^+(n)}\mathfrak{G}_1(\mcD')$$
    is set of all $\mcD\in Z_G^+(n+1)$ that satisfy the following:
    \begin{enumerate}
    \item $s_{n+1}$ is a critical point of $\mcD$.
    \item If $\mcD_{s_{n+1}}\subseteq \mcD_{s_n}$ and $|\mcD_{s_n}|>2$, then $\min(\mcD_{s_n})=\min(\mcD_r)$ for some $r\neq s_n$.
    \end{enumerate}
\end{lemma}

\begin{proof}
First, if $\mcD'\in Z_\Gamma^+(n)$ and $\mcD\in \mathfrak{G}_1(\mcD')$, then $s_{n+1}$ is critical point of $\mcD$ by definition.  Conversely, suppose $\mcD\in Z_\Gamma^+(n+1)$ with $s_{n+1}$ a critical point. Define the staircase diagrams
$$\widetilde\mcD:=\mcD\setminus\{B_{\mcD}\}$$
and
\begin{equation*}
    \overline{\mcD}:=\{\overline{B}\ |\ B\in\mcD\}
\end{equation*}
    where $\overline{\mcD}$ has the same covering relations as $\mcD$ and
    $$\overline{B}:=
        \begin{cases}
            B & \text{if}\ B\notin \mcD_{s_n}\cup\{B_{\mcD}\}\\
            B\setminus\{s_n\} & \text{if}\  B\in \mcD_{s_n}\setminus\{B_{\mcD}\}\\
            B\setminus\{s_{n+1}\}& \text{if}\ B=B_{\mcD}
        \end{cases}$$
It suffices to show that either $\widetilde\mcD$ or $\overline{\mcD}$ belongs to $Z_\Gamma^+(n).$  If $B_{\mcD}=\{s_{n+1}\}$, then $\widetilde\mcD\in Z_\Gamma^+(n)$ and hence we can assume $|B_{\mcD}|\geq 2$.  If $|\mcD_{s_n}|=2$, then $s_n$ is a critical point of $\widetilde\mcD$.  If $\widetilde\mcD\in Z_\Gamma^+(n)$, then we are done.  Otherwise, $\widetilde\mcD$ is not elementary and hence $s_{n-1}$ is also a critical point of $\widetilde\mcD$.  But this implies that $\overline{\mcD}\in Z_\Gamma^+(n)$.  Lastly, suppose $|\mcD_{s_n}|>2$.  Note that $\widetilde\mcD\in Z_\Gamma^+(n)$, but $\mcD\notin \mathfrak{G}_1(\widetilde\mcD)$ since $|\widetilde\mcD_{s_n}|\geq 2$.  Now $\overline{\mcD}$ is a valid staircase diagram since $\min(\mcD_{s_n})=\min(\mcD_r)$ for some $r\neq s_n$ and hence $\overline{\mcD}\in Z_\Gamma^+(n)$.  Moreover, $\mcD\in \mathfrak{G}_1(\overline{\mcD})$ which completes the proof.
\end{proof}

We end this section with two additional lemmas which apply only when $\Gamma$
is a tree. First, we can replace part (2) of Definition \ref{D:staircase} with
the following stronger condition:
\begin{lemma}\label{L:staircase_trees}
    Let $\mcD$ be a staircase diagram over a tree graph $\Gamma$.  Then $\mcD_s$ is a saturated chain for every $s\in S$.
\end{lemma}

\begin{proof}
    By part (2) Definition \ref{D:staircase} we have that $\mcD_s$ is a chain.  Suppose $\mcD_s$ is not a saturated chain.  Then there exists a saturated chain $B_0\prec\cdots\prec B_m$ in $\mcD$ where $B_0, B_m\in \mcD_s$ and $B_i\notin \mcD_s$ for $1\leq i< m$.  By part (1) of Definition \ref{D:staircase}, the set $B_0\cup\cdots\cup B_m$ forms a connected subgraph of $\Gamma$.   Thus $\Gamma$ has cycle which is a contradiction.
\end{proof}

Define $$Z_{\Gamma}^-(n):\{\mcD\in Z_\Gamma(n)\ |\ \text{$s_n$ is contained in a minimal block}\}.$$
It is easy to see that the sets $Z_\Gamma^{\pm}(n)$ are in natural bijection by $\mcD\mapsto \flip(\mcD)$.  Hence, any enumerative properties of $Z_\Gamma^+(n)$ apply to $Z_\Gamma^-(n)$.  If we further suppose that the set $Z_\Gamma^+(n)$ consists only of chains, then we have the following lemma.

\begin{lemma}\label{L:elementarychains_disjoint}
Let $\Gamma$ be a tree graph and suppose every $\mcD\in Z_\Gamma^+(n)$ is a chain.  Then the intersection $$Z_\Gamma^+(n)\cap Z_\Gamma^-(n)=\{\mcD\in Z_\Gamma^+(n)\ |\ \mcD=\mcD_{s_n}\}.$$
\end{lemma}

\begin{proof}
If $\mcD$ is a chain, we can write $\mcD=(B_1\prec\cdots \prec B_m)$.  If $\mcD\in Z_\Gamma^+(n)\cap Z_\Gamma^-(n)$, then $s_n\in B_1\cap B_m$.  Lemma \ref{L:staircase_trees} implies $\mcD_{s_n}$ is a saturated chain in $\mcD$.  Hence $\mcD=\mcD_{s_n}$.
\end{proof}

If $\Gamma$ is not a tree, then both Lemma \ref{L:staircase_trees} and
\ref{L:elementarychains_disjoint} are false. Lemma
\ref{L:elementarychains_disjoint} is also false without the assumption that
$\mcD$ is a chain. For the simply-laced finite type groups, the Coxeter-Dynkin
graph $\Gamma$ is a tree, and this will be crucial to proof of Theorem
\ref{T:main_gen_series} in the next three sections.

\section{Staircase diagrams of type A}\label{S:typeA}

Corollary \ref{C:simplylacedbij} implies that the number of smooth Schubert
varieties of type $A_n$ is precisely the number of staircase diagrams over the
Dynkin graph of type $A_n$.  Let $\Gamma$ be the Dynkin graph of type $A_1$
with $s=s_1$.  In the notation of the previous section, $\Gamma_n$ is the
Dynkin graph of type $A_n$, pictured below:
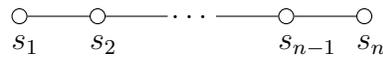
\begin{equation*}
    \begin{tikzpicture}[start chain]
        \dnode{1} \dnode{2} \dydots \dnode{n-1} \dnode{n}
    \end{tikzpicture}
\end{equation*}

If $\mcD$ is a staircase diagram of type $A_n$, then each $B\in \mcD$ has a
distinct left and right endpoint (if $B,B'\in\mcD$ share a common
endpoint then either $B$ contains $B'$, or vice-versa). The vertices of largest
and smallest index in $\mcD$ will be critical points of $\mcD$.

\begin{lemma}\label{L:chain_diagrams}
    Let $\mcD$ be a staircase diagram of type $A_n$ with full support.  If $\mcD$ is an elementary diagram, then $\mcD$ is an chain.
\end{lemma}

\begin{proof}
Since there are at most two critical points, $\mcD$ has a unique maximal and minimal block by Lemma \ref{L:extreme_crits}.  Thus $\mcD$ is a chain.
\end{proof}
Let $Z_A(n):=Z_\Gamma(n)$ denote the set of fully supported elementary diagrams
of type $A_n$, and set $Z_A^{\pm}(n):=Z_{\Gamma}^{\pm}(n)$.  Lemma
\ref{L:chain_diagrams} implies that $$Z_A(n)=Z_A^+(n)\cup Z_A^-(n).$$
Furthermore,  if $n\geq 3$, then Lemma \ref{L:elementarychains_disjoint}
implies that $Z_A^+(n)\cap Z_A^-(n)=\emptyset$ and hence $|Z_A(n)|=2|Z^+_A(n)|$.
Let
\begin{equation*}
    \Cat(t):=\sum_{n=0}^{\infty}\cc_n\, t^n=\frac{1-\sqrt{1-4t}}{2t}
\end{equation*}
denote the generating series of Catalan numbers, and set
\begin{equation*}
    \displaystyle A_Z(t):=\sum_{n=1}^{\infty} z_n\, t^n
\end{equation*}
where $z_n:=|Z_A(n)|$.


\begin{proposition}\label{P:chain_A_genseries}
    For $n=1,2$, we have $z_1=1$ and $z_2=3.$  If $n\geq 3$, then
    $z_n=2\cc_{n-2}$. Consequently
    $$A_Z(t)=t+t^2+2t^2\Cat(t).$$
\end{proposition}

\begin{proof}
    It easy to check that $z_1=1$, $z_2=3$ and $z_3=2$.  In particular, $Z_A^+(3)$ contains the single staircase diagram $$\mcG:=\{[s_1,s_2]\prec[s_2,s_3]\}.$$  For $n\geq 3$, we will prove $z_n=2\cc_{n-2}$ by showing $Z_A^+(n)=\mathfrak{G}_{n-3}(\mcG)$ and applying Proposition \ref{P:Catalan_gen} part (2).  Every $\mcD\in Z_A^+(n)$ satisfies part (1) of Lemma \ref{L:Catalan_almostsurjective} and part (2) of the lemma is vacuously true.  This implies
    $$Z_A^+(n+1)=\bigcup_{\mcD\in Z_A^+(n)}\mathfrak{G}_1(\mcD)$$ for $n\geq 3$ and $Z_A^+(n)=\mathfrak{G}_{n-3}(\mcG)$.  By Proposition \ref{P:Catalan_gen} part (2),
    \begin{equation*}
        A_Z(t)=t+3t^2+2\sum_{n=1}^\infty \cc_n\, t^{n+2}=t+t^2+2t^2\sum_{n=0}^\infty \cc_n\, t^n.
    \end{equation*}
\end{proof}
Define the generating series
\begin{equation}\label{Eq:genseries_Afull}
    \displaystyle \overline{A}(t):=\sum_{n=1}^\infty \overline{a}_n\, t^n
\end{equation}
where $\overline{a}_n$ denotes the number of staircase diagrams of type $A_n$ with full support.
\begin{proposition}\label{P:typeA_genseries_full}
    The generating series $\displaystyle\overline{A}(t)=\frac{t^2}{2t-A_Z(t)}$.
\end{proposition}

\begin{proof}
    Let $\mcD$ be a staircase diagram of type $A_n$ with full support.  Let $s_k,s_n$ be the critical points of $\mcD$ with the two largest indices.  We can write $\mcD$ as the union of two staircase diagrams $\mcD',\mcD''$ where we intersect each block of $\mcD$ with $[s_1,s_k]$ and $[s_k,s_n]$ respectively.  For example, intersecting $\mcD$ below with $[s_1,s_7]$ and $[s_7,s_{10}]$ gives the following diagrams $\mcD'$ and $\mcD''$.

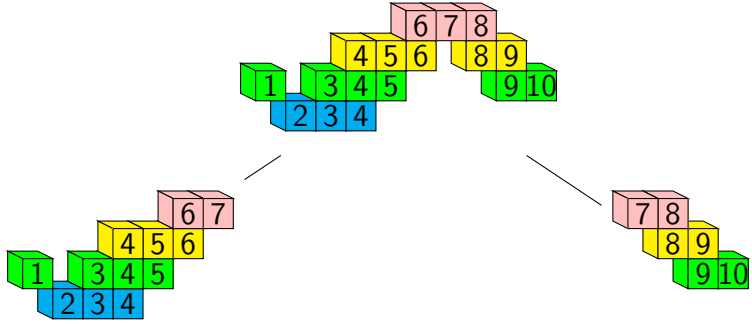
\begin{equation*}
\begin{tikzpicture}[level distance=25mm,
level 1/.style={sibling distance=75mm}]
    \node {\begin{tikzpicture}[scale=0.4]
            \planepartition{10}{
                {0,0,0,0,0,0,1,1,1,0},
                {1,1,0,0,0,1,1,1,0,1},
                {0,1,1,0,1,1,1},
                {0,0,1,1,1}}
            \end{tikzpicture}}
        child{ node{
        \begin{tikzpicture}[scale=0.4]
            \planepartition{7}{
                {0,0,0,1,1,1,0},
                {0,0,1,1,1,0,1},
                {0,1,1,1},
                {1,1,0,0}}
        \end{tikzpicture}}}
        child{ node {
        \begin{tikzpicture}[scale=0.4]
            \planepartition{10}{
                {0,0,0,0},
                {1,1,0,0},
                {0,1,1,0},
                {0,0,1,1}}
        \end{tikzpicture}}};
\end{tikzpicture}
\end{equation*}

    It is easy to see that $\mcD'$ is a staircase diagram of type $A_{n-k}$ with full support.  Furthermore,  $\mcD''$ is an elementary diagram with support $[s_k,s_n]$ by Lemma \ref{L:chain_diagrams}.  Hence
    \begin{equation*}
        \overline{a}_n=\sum_{k=1}^{n-1} \overline{a}_k z_{n+1-k}.
    \end{equation*}
    This implies
    \begin{align*}
        t^2&=\sum_{n=1}^\infty \left(\overline{a}_n -\sum_{k=1}^{n-1} \overline{a}_k z_{n+1-k}\right)t^{n+1}
        =2t\sum_{n=1}^\infty \overline{a}_n t^{n} -\sum_{n=1}^\infty\left(\sum_{k=1}^{n} \overline{a}_k z_{n+1-k}\right) t^{n+1}\\
        &=\overline{A}(t)(2t -A_Z(t))
    \end{align*}
\end{proof}
Define the generating series
\begin{equation}\label{Eq:genseries_A}
    A(t):=\sum_{n=0}^\infty a_n\, t^n.
\end{equation}
where $a_n$ denotes the total number of staircase diagrams of type $A_n$.  Here we set $a_0:=1$.

\begin{theorem}\label{T:typeA_genseries}
    The generating series $\displaystyle A(t)=\frac{1+\overline{A}(t)}{1-t-t\overline{A}(t)}$.
\end{theorem}

\begin{proof}
    Every staircase diagram is a disjoint union of staircase diagrams with connected support.  Hence
    \begin{equation*}
        A(t)=\sum_{n=1}^\infty (1+\overline{A}(t))^{n} t^{n-1}.
    \end{equation*}
\end{proof}

\section{Labelled staircase diagrams of type BC}\label{S:typeBC}

Let $\Gamma_n$ be the Dynkin graph of type $B_n$ or $C_n$.  Since Definition \ref{D:staircase} gives no consideration to double edges, staircase diagrams over $B_n$ or $C_n$ correspond precisely with staircase diagrams over $A_n$.

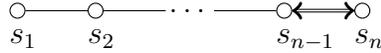
\begin{equation*}
    \begin{tikzpicture}[start chain]
        \dnode{1} \dnode{2} \dydots \dnode{n-1} \dnodedub{n}
    \end{tikzpicture}
\end{equation*}

In this section, we study and enumerate rationally smooth, almost-maximal labelled staircase diagrams of type $B_n$ and $C_n$.  As mentioned before, the Weyl groups of type $B$ and $C$ are isomorphic;we write $W$ for either Coxeter group.
While staircase diagrams only depend on the underlying graph with single edges, labellings depend on the corresponding Coxeter group. Hence edge labels of the Dynkin graph will play an important role.   The rationally smooth, almost maximal labellings of type $BC_n$ come in three types.

Let $\mcD$ be a staircase diagram of type $A_n$ and define the map $\lambda_1:\mcD\rightarrow W$ by

\begin{equation*}
\lambda_1(B):=
\begin{cases}
    u_B & \text{if $s_n\notin B$}\\
    s_r\cdots s_n u_{B\setminus\{s_n\}} & \text{if $s_n\in B$ where $B=[s_r,s_n]$}.
\end{cases}
\end{equation*}

The following lemma is from \cite[Proposition 5.4]{RS14} and \cite[Lemma 5.7]{RS14}.

\begin{lemma}\label{L:BC1_prop}
Suppose $s_n\in B$.  Then the following are true:
\begin{enumerate}
    \item $D_R(\lambda_1(B))=D_L(\lambda_1(B))=B\setminus\{s_n\}$
    \item $\supp(\lambda_1(B)u_J)=\supp(u_J\lambda_1(B))=B$ for any $J\subseteq B\setminus\{s_n\}$.
\end{enumerate}
\end{lemma}

\begin{lemma}\label{L:BC1_label}
    If $\mcD$ is a staircase diagram of type $A_n$, then the map $\lambda_1:\mcD\rightarrow W$ is a labelling of $\mcD$.
\end{lemma}

\begin{proof}
    If $s_n\notin\supp(\mcD)$, then $\lambda_1$ is the maximal labelling of $\mcD$.  If $s_n\in B$, then Lemma \ref{L:chain_diagrams} implies $s_n$ is a critical point of $\mcD$ and hence $J_R(B)$ and $J_L(B)$ are contained in $B\setminus\{s_n\}$.  Hence Lemma \ref{L:BC1_prop} implies $\lambda_1$ satisfies parts (1) and (2) of Definition \ref{D:labelling}. Part (3) also follows directly from Lemma \ref{L:BC1_prop}, completing the proof.
\end{proof}

Let $\mcD$ be a staircase diagram of type $A_n$ with full support and let $B_0$ denote the unique block in $\mcD_{s_n}$.  Recall the definition of maximal labelling from Definition \ref{D:maximal_labelling}.  It is easy to see that $\lambda_1$ is the maximal labelling of $\mcD$ if and only if $B_0=\{s_n\}$.  For $n\geq 2$, let $\overline{BC}_1(n)$ denote set of non-maximal $\lambda_1$-labelled staircase diagrams of type $A_n$ with full support.


\begin{proposition}\label{P:BC_label1_enum}
$|\overline{BC}_1(n)|=\overline{a}_n-2\overline{a}_{n-1}$ where $\overline{a}_n$, defined in Equation \eqref{Eq:genseries_Afull}, denotes the number of fully supported staircase diagrams of type $A_n$.
\end{proposition}

\begin{proof}
Let $B_0$ denote the unique block in $\mcD_{s_n}$.  If $\lambda_1$ is maximal, then $B_0=\{s_n\}$.  This implies that $\mcD\setminus\{B_0\}$ is a fully supported staircase diagram of type $A_{n-1}$ and hence $s_{n-1}$ is critical point of $\mcD$.  If $B_1$ denotes the unique block in $\mcD_{s_{n-1}}$, then we either have $B_1\prec B_0$ or $B_0\prec B_1$.  Hence, if $n\geq 2$, then there are $\overline{a}_n-2\overline{a}_{n-1}$ fully supported staircase diagrams of type $A_n$ where $\lambda_1$ is non-maximal.
\end{proof}

We consider the next two types of labellings together.  For any $B=[s_r,s_n]$ and $r\leq k<n$, we consider two special Weyl groups elements of $W$.  Define
$$u_B(k):=s_{k+1} s_{k+2}\cdots s_n s_{n-1}\cdots s_r u_{B\setminus\{s_r\}}$$
and
$$u'_B(k):=u^{B\setminus\{s_r,s_k+1\}}_{B\setminus\{s_r\}}s_rs_{r+1}\cdots s_k u_{B\setminus\{s_k\}},$$
where $u^{B\setminus\{s_r,s_{k+1}\}}_{B\setminus\{s_r\}}$ denotes the longest element of $W_{B\setminus\{s_r\}}\cap W^{B\setminus\{s_r,s_k+1\}}$.

Let $\mcD$ be a staircase diagram of type $A_n$ and define the maps $\lambda^k_2:\mcD\rightarrow W$ and $\lambda^k_3:\mcD\rightarrow W$ by

\begin{equation*}
\lambda^k_2(B):=
\begin{cases}
    u_B & \text{if $s_n\notin B$}\\
    u_B(k) & \text{if $s_n\in B$}
\end{cases}
\qquad\text{and}\qquad
\lambda^k_3(B):=
\begin{cases}
    u_B & \text{if $s_n\notin B$}\\
    u'_B(k) & \text{if $s_n\in B$}.
\end{cases}
\end{equation*}
where $B=[s_r,s_n]$ with $r\leq k< n$ (so in particular, $|B|\geq 2$).  Unlike the labelling $\lambda_1$, the maps $\lambda^k_2, \lambda^k_3$ are not always valid labellings of $\mcD$.  The following lemma is from \cite[Proposition 5.4]{RS14} and \cite[Lemma 5.6]{RS14}.

\begin{lemma}\label{L:BC2_prop}
Let $B=[s_r,s_n]$ with $r\leq k< n$.  Then the following are true:
\begin{enumerate}
    \item $D_R(u_B(k))=D_L(u'_B(k))=B\setminus\{s_r\}$
    \item $D_L(u_B(k))=D_R(u'_B(k))=B\setminus\{s_k\}$
    \item $\supp(u_B(k)u_J)=\supp(u_J u'_B(k))=B$ for any $J\subseteq B\setminus\{s_r\}$
    \item $\supp(u'_B(k)u_J)=\supp(u_J u_B(k))=B$ for any $J\subseteq B\setminus\{s_k\}$.
\end{enumerate}
\end{lemma}

Clearly if $s_n\notin\supp(\mcD)$, then $\lambda_2^k,\lambda_3^k$ are both maximal labellings of $\mcD$.  If $s_n\in\supp(\mcD)$, then $|\mcD_{s_n}|=1$ and we have the following lemma.

\begin{lemma}\label{L:BC2_label}
    Let $\mcD$ be a staircase diagram of type $A_n$ with $s_n\in\supp(\mcD)$ and let $B=[s_r,s_n]\in \mcD_{s_n}$.   Then $\lambda_2^k$ is a labelling of $\mcD$ if and only if $J_R(B)=\emptyset$ and $J_L(B)\subseteq [s_r,s_{k-1}]$.

    Moreover, $\lambda_2^k$ is a labelling of $\mcD$ if and only if $\lambda_3^k$ is a labelling of $\flip(\mcD)$.
\end{lemma}

\begin{proof}
    The fact that $\lambda_2^k$ is a labelling of $\mcD$ if and only if $\lambda_3^k$ is a labelling of $\flip(\mcD)$ is a direct consequence of Lemma \ref{L:BC2_prop}.  Suppose $\lambda_2^k$ is a labelling of $\mcD$.  Then $J_R(B)\subseteq D_R(\lambda^k_2(B))$ and $J_L(B)\subseteq D_L(\lambda^k_2(B))$ for $B=[s_r,s_n]$.  Lemma \ref{L:BC2_prop} implies that $B$ is minimal in $\mcD_{s_r}$ and thus $B$ is minimal in $\mcD$.  Hence $J_R(B)=\emptyset$.  Conversely, if $J_R(B)=\emptyset$, then the condition that $J_R(B)\subseteq D_R(\lambda^k_2(B))$ is trivial.  A similar argument gives that $J_L(B)\subseteq [s_r,s_{k-1}]$.  Finally, Lemma \ref{L:BC2_prop} implies that part (3) of Definition \ref{D:staircase} is always satisfied.  This completes the proof.
\end{proof}

\begin{lemma}\label{L:BC2_intersection}
    Let $B=[s_r,s_n]$ with $r\leq k< n$.  Then $u_B(k)=u'_B(k')$ if and only if $k=k'=r$.
\end{lemma}

\begin{proof}
If either $k$ or $k'$ is not equal to $r$, then $u_B(k)\neq u'_B(k')$ by Lemma \ref{L:BC2_prop}.  If $k=k'=r$, then $$u^{B\setminus\{s_r,s_{r+1}\}}_{B\setminus\{s_r\}}=s_{r+1}s_{r+2}\cdots s_ns_{n-1}\cdots s_{r+1}$$ and hence $u_B(r)= u'_B(r)$.
\end{proof}

If $B=[s_r,s_n]\in \mcD_{s_n}$, then Lemma \ref{L:BC2_intersection} implies that $\Lambda(\mcD,\lambda_2^k)=\Lambda(\mcD,\lambda_3^{k'})$ if and only if $k=k'=r$.  Let $\overline{BC}^k_2(n)$ denote set of labelled staircase diagrams of type $A_n$ with full support where the labelling is given by either $\lambda_2^k$ or $\lambda_3^k$.

\begin{proposition}\label{P:BC_label23_enum}
$|\overline{BC}^1_2(n)|=1$, and if $k\geq 2$ then
$$\displaystyle |\overline{BC}^k_2(n)|=1+\overline{a}_{k}+2\sum_{\ell=1}^{k-2} \overline{a}_{\ell}.$$
\end{proposition}

\begin{proof}
    First, if $k=1$ then $\mcD=\{[s_1,s_n]\}$ is the unique unlabelled staircase diagram in $\overline{BC}^1_2(n)$.  By Lemma \ref{L:BC2_intersection}, the labelling $\lambda_2^1=\lambda_3^{1}$, and hence $|\overline{BC}^1_2(n)|=1$.   For $k\geq 2$, let $\mcD\in\overline{BC}^k_2(n)$ with $B\in\mcD_{s_n}$.  Lemma \ref{L:BC2_label} implies that $s_k,\ldots, s_n\in B$ are critical points of $\mcD$.  Hence $\mcD$ can be identified with a fully supported staircase diagram of type $A_k$ by replacing the block $B$ with $B\setminus[s_{k+1},s_n]$ in $\mcD$.  Let $\overline{A}(k)$ denote the set of fully supported (unlabelled) staircase diagrams of type $A_k$ and define the sets
    \begin{align*}
    F_{k}^+&:=\{\mcD\in \overline{A}(k)\ |\ \text{$s_{k}\in B$ with $J_L(B)=\emptyset$}\}\\
    F_{k}^-&:=\{\mcD\in \overline{A}(k)\ |\ \text{$s_{k}\in B$ with $J_R(B)=\emptyset$}\}
    \end{align*}
    Lemma \ref{L:BC2_label} implies the number of $\lambda_2^k$-labelled staircase diagrams in $BC^k_2(n)$ is $|F_{k}^-|$.  Similarly the number of $\lambda_3^k$-labelled staircase diagrams in $BC^k_2(n)$ is $|F_{k}^+|$.  Hence
    $$|F_{k}^+|+|F_{k}^-|=|A(k)|+|F_k^+\cap F_k^-|=\overline{a}_{k}+|F_k^+\cap F_k^-|.$$
    If $\mcD\in F_k^+\cap F_k^-$ with $s_{k}\in B$, then every element of $B$ is a critical point.  If $B=[s_{\ell},s_{k}]$, then $\mcD$ can be identified with a fully supported staircase diagram of type $A_{\ell-1}$ by removing $B$ from $\mcD$.   If $\ell\geq 2$, let $B'$ denote the unique block in $\mcD_{\ell-1}$.  Then either $B'\prec B$ or $B\prec B'$.  This gives
    $$|F_k^+\cap F_k^-|=1+2\sum_{\ell=1}^{k-1} \overline{a}_{\ell}.$$
    Lemma \ref{L:BC2_intersection} implies that the labelling $\lambda_2^k=\lambda_3^k$ on $\mcD\in\overline{BC}^k_2(n)$ if and only if $B=[s_k,s_n]\in\mcD_{s_n}$.  Since there are $2\overline{a}_{k-1}$ staircase diagrams in $\overline{BC}^k_2(n)$ that satisfy this condition, we have
    $$|\overline{BC}^k_2(n)|=|F_{k}^+|+|F_{k}^-|-2\overline{a}_{k-1}=1+\overline{a}_{k}+2\sum_{\ell=1}^{k-2} \overline{a}_{\ell}.$$
\end{proof}

If $w\in W$, let $X^B(w)$ and $X^C(w)$ denote the corresponding Schubert variety of type $B_n$ and $C_n$ respectively.  Let $\lambda_0:\mcD\rightarrow W$ denote the maximal labelling.  The geometric significance of labellings $\lambda_0, \lambda_1, \lambda_2^k$ and $\lambda_3^k$ is given by the following proposition, which follows from \cite[Theorem 3.8]{RS14} and Corollary \ref{C:ratsmoothbij}.

\begin{proposition}\label{P:BC_StaircaseSchubert}
    Let $W$ denote the Weyl group of type $B_n$ or $C_n$ and let $\mcD$ be a staircase diagram of type $A_n$.  Then $\lambda:\mcD\arr W$ is an almost-maximal, rationally smooth labelling if and only if $\lambda$ equals one of $\lambda_0, \lambda_1, \lambda_2^k$ or $\lambda_3^k$.  Moreover, if $w\in W$, then following are true:
    \begin{enumerate}
        \item The Schubert variety $X^B(w)$ is smooth if and only if $w=\Lambda(\mcD,\lambda_0)$, or $w=\Lambda(\mcD,\lambda_1)$ for some staircase diagram $\mcD$ of type $A_n.$
        \item The Schubert variety $X^C(w)$ is smooth if and only if $w=\Lambda(\mcD,\lambda_0)$, or $w=\Lambda(\mcD,\lambda_2^k)$ or $\Lambda(\mcD,\lambda_3^k)$ for some staircase diagram $\mcD$ of type $A_n$ with $1\leq k< n.$
        \item The Schubert varieties $X^B(w),X^C(w)$ are rationally smooth if and only if at least one of $X^B(w)$ or $X^C(w)$ is smooth.
    \end{enumerate}
\end{proposition}

Define the generating series
\begin{equation}
\overline{B}(t):=\sum_{n=1}^{\infty} \overline{b}_n\, t_n\qquad\text{and}\qquad \overline{C}(t):=\sum_{n=1}^{\infty} \overline{c}_n\, t_n,
\end{equation}
where $\overline{b}_n$ and $\overline{c}_n$ denote the number of fully supported smooth Schubert varieties of types $B_n$ and $C_n$ respectively.
We also define
\begin{equation}
\overline{BC}(t):=\sum_{n=1}^{\infty} \overline{bc}_n\, t_n,
\end{equation}
where $\overline{bc}_n$ denotes the number of fully supported rationally smooth Schubert varieties of either types $B_n$ or $C_n$.

\begin{proposition}\label{P:typeBC_genseries_full}
The above generating series satisfy the following identities:
\begin{align}
\overline{B}(t)&=(2-2t)\overline{A}(t)-t \label{Eq:genseries_Bfull},\\
\overline{C}(t)&=\frac{\overline{A}(t)}{1-t}+\frac{t^3(1+2\overline{A}(t))}{(1-t)^2}\label{Eq:genseries_Cfull}, \text{ and }\\
\overline{BC}(t)&=\overline{B}(t)+\overline{C}(t)-\overline{A}(t).\label{Eq:genseries_BCfull}
\end{align}
\end{proposition}

\begin{proof}
    Proposition \ref{P:BC_StaircaseSchubert} implies that $X^B(w)$ is smooth if and only if $w=\Lambda(\mcD,\lambda_0)$ or $w=\Lambda(\mcD,\Lambda_1)$ for some staircase diagram $\mcD$ of type $A_n.$  By Proposition \ref{P:BC_label1_enum}, we have
        $$b_n=\overline{a}_n + |\overline{BC}_1(n)|=\overline{a}_n+(\overline{a}_n-2\overline{a}_{n-1})$$
    for all $n \geq 2$.  Hence
        $$\overline{B}(t)=\overline{A}(t)+(\overline{A}(t)-t-2t\overline{A}(t))=(2-2t)\overline{A}(t)-t,$$
    which proves Equation \eqref{Eq:genseries_Bfull}.  For type $C_n$, Proposition \ref{P:BC_StaircaseSchubert} implies that $X^C_w$ is smooth if and only if $w=\Lambda(\mcD,\lambda_0)$, or $w=\Lambda(\mcD,\lambda_2^k)$ or $\Lambda(\mcD,\lambda_3^k)$ for some staircase diagram $\mcD$ of type $A_n$ and $1\leq k< n.$  Proposition \ref{P:BC_label23_enum} implies that
    \begin{align*}
        \overline{c}_n&=\overline{a}_n + \sum_{k=1}^{n-1}|\overline{BC}_2(n)|\\
        &= \overline{a}_n + 1+ \sum_{k=2}^{n-1} \left(1+\overline{a}_{k}+2\sum_{\ell=1}^{k-2} \overline{a}_{\ell}\right)\\
        &= \left(\sum_{k=1}^n \overline{a}_k\right) + (n-2) + 2\left(\sum_{k=1}^{n-3}\sum_{\ell=1}^{k} \overline{a}_{\ell}\right).
    \end{align*}
    This gives
    \begin{align*}
        \overline{C}(t)&= \left(\sum_{n=1}^{\infty}\sum_{k=1}^n \overline{a}_k\, t^n\right) + \left(\sum_{n=2}^{\infty}(n-2)\, t^n\right) + 2 \left(\sum_{n=2}^{\infty} \sum_{k=1}^{n-3}\sum_{\ell=1}^{k} \overline{a}_{\ell}\, t^n\right).\\
        &= \frac{\overline{A}(t)}{1-t} + \frac{t^3}{(1-t)^2} + \frac{2t^3\overline{A}(t)}{(1-t)^2}
    \end{align*}
    which implies Equation \eqref{Eq:genseries_Cfull}.  Finally, Equation \eqref{Eq:genseries_BCfull} follows directly from part (3) of Proposition \ref{P:BC_StaircaseSchubert}.
\end{proof}

Recall from the introduction that
\begin{equation}
    B(t):=\sum_{n=1}^{\infty} b_n\, t_n\qquad\text{and}\qquad C(t):=\sum_{n=1}^{\infty} c_n\, t_n,
\end{equation}
where $b_n$ and $c_n$ denote the number of smooth Schubert varieties of types $B_n$ and $C_n$ respectively, and that
\begin{equation}
BC(t):=\sum_{n=1}^{\infty} bc_n\, t_n,
\end{equation}
where $bc_n$ denotes the number of rationally smooth Schubert varieties of either types $B_n$ or $C_n$. Set $b_0=c_0=bc_0=1$, and let $A(t)$ continue to denote the generating series for staircase
diagrams of type $A$.

\begin{theorem}\label{T:typeBC_genseries}
    The above generating series satisfy the following identities:
    \begin{align}
        B(t)&=(1+tA(t))(1+\overline{B}(t)), \label{Eq:genseries_B}\\
        C(t)&=(1+tA(t))(1+\overline{C}(t)), \text{ and}\label{Eq:genseries_C}\\
        BC(t)&=(1+tA(t))(1+\overline{BC}(t)).\label{Eq:genseries_BC}
    \end{align}
\end{theorem}

\begin{proof}
The proof is the same for each of $B(t), C(t)$ and $BC(t)$, so we focus on the series $B(t)$.  Suppose that $w\in W$ corresponds to a smooth Schubert variety of type $B_n.$  First, if $\supp(w)=[s_1,s_n]$, then by Proposition \ref{P:typeBC_genseries_full}, the generating series for these elements is $\overline{B}(t).$  Now suppose that $\supp(w)\neq [s_1,s_n]$ and let $k$ denote the largest index for which $s_k\notin \supp(w).$  If $k=n$, then the generating series for elements of this type is $tA(t)$.  If $k<n$, then $w$ corresponds to a staircase diagram of type $A_{k-1}$ and a labelled staircase diagram enumerated by $\overline{B}(t)$.  Hence the generating series for these elements is $tA(t)\overline{B}(t)$.  Thus
$$B(t)=1+\overline{B}(t)+tA(t)+ tA(t)\overline{B}(t).$$  The proof is the same for $C(t)$ and $BC(t)$, where we replace $\overline{B}(t)$ by $\overline{C}(t)$ and $\overline{BC}(t)$ respectively.
\end{proof}


\begin{rmk}\label{R:BC_smooth}
Proposition \ref{P:BC_StaircaseSchubert} implies that the generating series for the number of $w\in W$ for which $X^B(w)$ and $X^C(w)$ are both smooth is precisely $A(t)$.  In particular, $X^B(w)$ and $X^C(w)$ are both smooth if and only if $w=\Lambda(\mcD,\lambda_0)$ for some staircase diagram of type $A$.  In this case, both of these Schubert varieties decompose as (possibly different) iterated fiber bundles of Grassmannian flag varieties.
\end{rmk}

\section{Staircase diagrams of type D}\label{S:typeD}

As in the type $A$ case, Corollary \ref{C:simplylacedbij} implies that the number of smooth Schubert varieties of type $D_n$ is precisely the number of staircase diagrams over the Dynkin graph of type $D_n$.  Let $\Gamma$ be the Dynkin graph of type $D_3$ with vertices $S=\{s_1, s_2, s_3\}$, and set $s=s_3$. In the notation of Section \ref{S:staircasecatalan}, $\Gamma_n$ is the Dynkin graph of type $D_n$, pictured below:


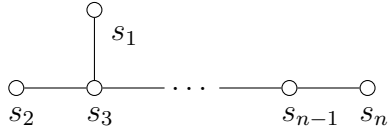
\begin{equation*}
    \begin{tikzpicture}
        \begin{scope}[start chain]
            \dnode{2} \dnode{3} \node[chj,draw=none] {\dots}; \dnode{n-1} \dnode{n}
        \end{scope}
        \begin{scope}[start chain=br going above]
            \chainin(chain-2); \dnodebr{1}
        \end{scope}
    \end{tikzpicture}
\end{equation*}

When analyzing elementary staircase diagrams of type $D$, it is convenient to treat $s_3$ as ``leaf" in the case where $n=3$.  Under this convention, all Dynkin graphs of type $D$ have three leaves.  Let $Z_D(n):=Z_\Gamma(n)$ denote the fully supported elementary staircase diagrams of type $D_n$.  Unlike in the type $A$ case, elementary staircase diagrams of type $D$ are not necessarily chains.  For example, we have the following elementary staircase diagram of type $D_8$:
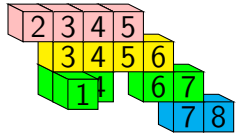
\begin{equation*}
    \begin{tikzpicture}[scale=0.4]
        \planepartitionD{8}{
            {1,1,0,0,0,0,0},
            {0,1,1,0,1,2,0},
            {0,0,1,1,1,1,0},
            {0,0,0,1,1,1,1}}
    \end{tikzpicture}
\end{equation*}

While elementary staircase diagrams of type $D$ may not be chains, they do have similar combinatorial structure to those of type $A$.  The next lemma follows immediately from property (4) of Definition \ref{D:staircase}.

\begin{lemma}\label{L:D_crit_pt_bound}
    If $\mcD\in Z_D(n)$, then $\mcD$ has at least 2 critical points with $|\mcD_{s_i}|\leq 2$ for $i=1,2,n$.  Furthermore, if $s_n$ is not a critical point, then $\mcD$ is either $([s_2,s_n]\prec [s_1,s_n])$ or $([s_1,s_n]\prec [s_2,s_n])$.
\end{lemma}

\begin{example}\label{Ex:typeD_6}
    We illustrate Lemma \ref{L:D_crit_pt_bound} by three examples in $Z_D(6)$.  The critical sets are $\{s_1,s_6\}$, $\{s_1,s_2\}$ and $\{s_1,s_2,s_6\}$ respectively.
    
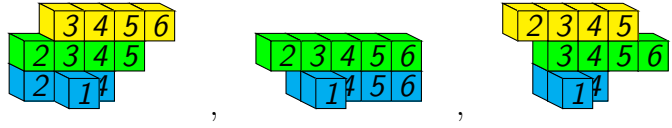
\begin{equation*}
        \begin{tikzpicture}[scale=0.4]
            \planepartitionD{6}{
               {0,0,1,2,1},
                {0,1,1,1,1},
                {1,1,1,1,0}}
        \end{tikzpicture}\quad, \quad
        \begin{tikzpicture}[scale=0.4]
            \planepartitionD{6}{
                {1,1,1,2,0},
                {1,1,1,1,1}}
        \end{tikzpicture}\quad, \quad
        \begin{tikzpicture}[scale=0.4]
            \planepartitionD{6}{
                {0,0,1,2,0},
                {1,1,1,1,0},
                {0,1,1,1,1}}
        \end{tikzpicture}
    \end{equation*}
\end{example}

Let $Z_D^{\pm}(n):=Z_{\Gamma}^{\pm}(n)$ and note that $$Z_D^+(n)\cup Z_D^-(n)\subset Z_D(n),$$ although equality doesn't hold as in the type $A$ case (see Example \ref{Ex:typeD_6}).  To enumerate $Z_D(n)$, we partition

\begin{equation}\label{Eq:typeD_partition}
Z_D(n)=Z_{D1}(n)\sqcup Z_{D2}(n)\sqcup Z_{D3}(n)
\end{equation} where

\begin{align*}
Z_{D1}(n)&:=\{\mcD\in Z_D(n)\ |\ \mcD_{s_1}\cap\mcD_{s_2}\neq\emptyset \} \text{ and}\\
Z_{D2}(n)\sqcup Z_{D3}(n)&:=\{\mcD\in Z_D(n)\ |\ \mcD_{s_1}\cap\mcD_{s_2}=\emptyset\}.
\end{align*}

Lemma \ref{L:D_crit_pt_bound} implies that if $\mcD\in  Z_{D2}(n)\sqcup Z_{D3}(n)$, then $\mcD$ has unique blocks $B_0'$ and $B_0''$ in $\mcD_{s_1}$ and $\mcD_{s_2}$ respectively.  To distinguish the sets $Z_{D2}(n)$ and $Z_{D3}(n)$, we consider the relationship between these two blocks.
Define \begin{align*}
Z_{D2}(n)&:=\{\mcD\in Z_{D}(n)\ |\ \text{$\{B_0',B_0''\}$ is saturated} \} \text{ and}\\
Z_{D3}(n)&:=\{\mcD\in Z_{D}(n)\ |\ \text{$\{B_0',B_0''\}$ is not saturated} \}
\end{align*}

Examples of staircase diagrams in each of $Z_{D1}(n)$, $Z_{D2}(n)$, and $Z_{D3}(n)$ are given in Example \ref{Ex:typeD_6}.  We enumerate the sets $Z_{D1}(n)$, $Z_{D2}(n)$, and $Z_{D3}(n)$ separately.

\subsection{Enumerating $Z_{D1}(n)$:}
Suppose that $\mcD\in Z_{D1}(n)$, so that $\mcD$ has an element containing both $s_1$ and $s_2$.  Lemma \ref{L:D_crit_pt_bound} implies that $\mcD_{s_1}\cap\mcD_{s_2}$ contains a unique block $B_0\in\mcD$.
\begin{lemma}\label{L:D_case1_chain}
    Let $\mcD\in Z_{D1}(n)$.  Then $B_0$ is an extremal block of $\mcD$.  Moreover $\mcD$ is a chain.
\end{lemma}
\begin{proof}
      Since $B_0$ is connected, we necessarily have $s_3\in B_0$.  Without loss of generality, assume that $s_1$ is a critical point of $\mcD$.  Define $\mcD'$ by replacing $B_0$ with $B_0\setminus\{s_1\}$. If $s_2$ is a critical point of $\mcD$, then $\mcD'$ is a elementary staircase diagram of type $A$ with critical set $\{s_2,s_n\}$.  By Lemma \ref{L:chain_diagrams}, $\mcD'$ is a chain and hence $\mcD$ must also be a chain.  Since $s_{2}\in B_0$, we must have that $B_0$ is extremal in $\mcD$.  If $s_2$ is not a critical point of $\mcD$, then let $B_1$ denote the unique block in $\mcD_{s_2}\setminus \{B_0\}$.  Define $\mcD''$ by replacing $B_1$ with $B_1\setminus\{s_2\}$ in $\mcD'$.  If $|\mcD_{s_3}|=2$, then $\mcD''$ is an elementary chain of type $A$ and the lemma follows by the previous argument.  Lastly, if $|\mcD_{s_3}|>2$ and $s_2$ is not a critical point of $\mcD$, then $\mcD\setminus\{B_0\}$ is an elementary chain of type $A$.  Since $\mcD_{s_2}$ is saturated, either $B_0$ or $B_1$ is extremal in $\mcD$. If $B_0$ is extremal, then we are done.  If $B_1$ is extremal, then $B_1$ contains a critical point by Lemma \ref{L:extreme_crits}, implying that $s_n\in B_1$.  Since $B_1$ is connected, we have $B_1=[s_2,s_n]$ and thus $\mcD=\{B_0,B_1\}$.  Since $\mcD$ has only two elements, $B_0$ is again extremal.  This completes the proof.
\end{proof}

Lemma \ref{L:D_case1_chain} implies that $Z_{D1}(n)\subset Z_D^+(n)\cup Z_D^-(n)$.  In particular, if $\mcD\in Z_{D1}(n)\cap Z_D^+(n)$, then $B_0$ must be the unique minimal block.  If $n\geq 4$, then Lemmas \ref{L:elementarychains_disjoint} and \ref{L:D_case1_chain} imply that

\begin{equation}\label{Eq:typeD2.1}
|Z_{D1}(n)|=2|Z_{D1}(n)\cap Z_D^+(n)|.
\end{equation}

\begin{proposition}\label{P:D_case1}
    $|Z_{D1}(3)|=1$, and if $n\geq 4$ then $|Z_{D1}(n)|=4\cc_{n-2}-2\cc_{n-3}$.

\end{proposition}

\begin{proof}
     First, note that $Z_{D1}(3)$ contains the single staircase diagram $\{\{s_1,s_2,s_3\}\}$ and hence $|Z_{D1}(3)|=1$.   By Equation \eqref{Eq:typeD2.1}, it suffices to show that $|Z_{D1}(n)\cap Z_D^+(n)|=2c_{n-1}-c_{n-2}$.  Recall that if $\mcD\in Z_D^+(n)$, then $B_{\mcD}\in\mcD$ denotes the unique block containing $s_n$.  The set $Z_{D1}(4)\cap Z_D^+(4)$ has three staircase diagrams.  In particular, we have $$\mcG:=\{\{s_1,s_2,s_3\}\prec[s_3,s_4]\},$$ which has $|B_{\mcG}|=2$, and the two diagrams $$\mcG':=\{\{s_1,s_2,s_3\}\prec[s_1,s_4]\}\quad\text{and}\quad\mcG'':=\{\{s_1,s_2,s_3\}\prec[s_2,s_4]\},$$
    which have $|B_{\mcG'}|=|B_{\mcG''}|=3$.  Since parts (1) and (2) of Lemma \ref{L:Catalan_almostsurjective} are true for all $\mcD\in Z_{D1}(n)\cap Z_D^+(n)$, we have
    $$Z_{D1}(n)\cap Z_D^+(n)=\mathfrak{G}_{n-4}(\mcG)\cup\mathfrak{G}_{n-4}(\mcG')\cup\mathfrak{G}_{n-4}(\mcG'').$$
    Proposition \ref{P:Catalan_gen} parts (3) and (4) imply that
    \begin{align*}
        |Z_{D1}(n)\cap Z_D^+(n)|&=|\mathfrak{G}_{n-4}(\mcG)\cup\mathfrak{G}_{n-4}(\mcG')\cup\mathfrak{G}_{n-4}(\mcG'')|\\
        &=|\mathfrak{G}_{n-4}(\mcG)\cup\mathfrak{G}_{n-4}(\mcG')|+|\mathfrak{G}_{n-4}(\mcG)\cup\mathfrak{G}_{n-4}(\mcG'')|-|\mathfrak{G}_{n-4}(\mcG)|\\
        &=2\cc_{n-2}-\cc_{n-3}
    \end{align*}
\end{proof}

\subsection{Enumerating $Z_{D2}(n)$:}  Recall that if $\mcD\in  Z_{D2}(n)$, then $\mcD$ has unique blocks $B_0'$ and $B_0''$ in $\mcD_{s_1}$ and $\mcD_{s_2}$ respectively.  Since either $B_0'\prec B_0''$ or $B_0''\prec B_0'$, we can define the set
\begin{align*}
Z^{\circ}_{D2}(n)&:=\{\mcD\in Z_{D2}(n)\ |\ B_0'\prec B_0''\}
\end{align*}
It is easy to see that $|Z_{D2}(n)|=2|Z^{\circ}_{D2}(n)|$ and hence we focus on enumerating $Z_{D2}^{\circ}(n)$.

\begin{lemma}\label{L:D_case2.1_chain}
    If $\mcD\in Z^{\circ}_{D2}(n)$, then $\mcD$ is a chain.
\end{lemma}
\begin{proof}
    If $s_n$ is not a critical point of $\mcD$ then we are done by Lemma \ref{L:D_crit_pt_bound}.  Otherwise, let $B_{\mcD}$ denote the unique block in $\mcD_{s_n}$.  The critical points of $\mcD$ are contained in $B_0',B_0''$, and $B_{\mcD}$.  Without loss of generality suppose that $B_0'\setminus\{s_1\}\subset B_0''$ and let $\mcD'=\mcD\setminus\{B_0'\}$.  Since $\{B_0',B_0''\}$ is saturated, the critical points of $\mcD'$ are contained in $B_0''$ and $B_{\mcD}$.  Since all extremal blocks contain critical points, $\mcD'$ is a chain and hence $\mcD$ is a chain.
\end{proof}
Several staircase diagrams in $Z^{\circ}_{D2}(n)$ are illustrated in Example \ref{Ex:Catalan_genD}.  Lemma \ref{L:D_case2.1_chain} implies that $Z^{\circ}_{D2}(n)\subset Z^+_D(n)\cup Z^-_D(n)$.  Note that the intersection $$Z^{\circ}_{D2}(n)\cap Z^+_D(n)\cap Z^-_D(n)$$ contains the single staircase diagram $\{[s_1,s_n]\prec[s_2,s_n]\}$ and hence
$$|Z^{\circ}_{D2}(n)|=2|Z^{\circ}_{D2}(n)\cap Z^+_D(n)|-1.$$

\begin{proposition}\label{P:D_case2.1}
    $|Z^{\circ}_{D2}(3)|=3$, and if $n\geq 4$ then $\displaystyle |Z^{\circ}_{D2}(n)|=1+2\sum_{k=0}^{n-2}\cc_k$.
\end{proposition}
\begin{proof}
    We will show that $Z^{\circ}_{D2}(n)\cap Z^+_D(n)$ is generated by applying $\mathfrak{G}_p$ to particular staircase diagrams.  Consider the diagrams
    $$\mcG':=\{\{s_1\}\prec[s_2,s_3]\}\in Z^{\circ}_{D2}(3)\cap Z^+_D(3)$$
    and
    $$\mcG'':=\{[s_1,s_3]\prec[s_2,s_3]\prec[s_3,s_4]\}\in Z^{\circ}_{D2}(4)\cap Z^+_D(4).$$
    For any $n\geq 3$ we also define $$\mcG_n:=\{[s_1,s_n]\prec[s_2,s_n]\}\in Z^{\circ}_{D2}(n)\cap Z^+_D(n).$$  Note that $\mcG_4$ is given in Example \ref{Ex:Catalan_genD}.  Observe that $Z^{\circ}_{D2}(3)\cap Z^+_D(3)$ contains only the diagrams $\mcG'$ and $\mcG_3$, so $|Z^{\circ}_{D2}(3)|=3$.  For $n\geq 4$, the diagram $\mcG_n$ is the only element in $Z^{\circ}_{D2}(n)\cap Z^+_D(n)$ for which $s_n$ is not a critical point.  Lemma \ref{L:Catalan_almostsurjective} implies that
    $$\mathfrak{G}_{n-3}(\mcG')\cup\mathfrak{G}_{n-2}(\mcG'')\cup\bigcup_{k=3}^{n}\mathfrak{G}_{n-k}(\mcG_k)= Z^{\circ}_{D2}(n)\cap Z^+_D(n).$$
    If $\mcD=(B_1\prec\cdots\prec B_m)\in Z^{\circ}_{D2}(n)\cap Z^+_D(n)$ with $|B_m|=1$, then $\mcD$ is the unique element of $\mathfrak{G}_{1}(\mcG_{n-1})$.  Proposition \ref{P:Catalan_gen} now implies that
    $$|Z^{\circ}_{D2}(n)\cap Z^+_D(n)|=\cc_{n-2}+\cc_{n-3}+\left(1+\sum_{k=0}^{n-4}\cc_k\right)=1+\sum_{k=0}^{n-2}\cc_k.$$
\end{proof}


\subsection{Enumerating $Z_{D3}(n)$}

Before enumerating $Z_{D3}(n)$, we look take a closer look at the relationship between $Z_{D2}(n)$ and $Z_{D3}(n)$.

\begin{lemma}\label{L:D_case2_saturation}
    If $\mcD\in  Z_{D2}(n)\sqcup Z_{D3}(n)$ and $n\geq 4$, then $B_0'$ and $B_0''$ are comparable.  Furthermore, one of the following is true:
    \begin{enumerate}
        \item The set $\{B_0',B_0''\}$ forms a saturated chain in $\mcD$ and hence $\mcD\in Z_{D2}(n)$.
        \item There is a unique block $B_0\in\mcD_{s_3}$ such that either $\{B_0'\prec B_0\prec B_0''\}$ or $\{B_0'\succ B_0\succ B_0''\}$ form a saturated chain in $\mcD$.
    \end{enumerate}
\end{lemma}
\begin{proof}
    Note that $|\mcD_{s_3}|=2$ or $3$ since $n\geq 4$.  If $|\mcD_{s_3}|=3$, then $B_0', B_0''\in\mcD_{s_3}$.  Lemma \ref{L:staircase_trees} implies that $\mcD_{s_3}$ is a saturated chain and the lemma follows.  If $|\mcD_{s_3}|=2$, then without loss of generality assume that $B_0'\in \mcD_{s_3}$ and $B_0''\notin\mcD_{s_3}$.  Then $B_0''=\{s_2\}$ is only comparable to elements of $\mcD_{s_3}$.  Since $\mcD_{s_2}\cup \mcD_{s_3}$ is a chain and $\mcD_{s_3}$ is saturated, we have that $\mcD_{s_2}\cup \mcD_{s_3}$ is also saturated.  This proves the lemma.
\end{proof}

If $n=3$, then Lemma \ref{L:D_case2_saturation} may fail. In particular, $Z_{D3}(3)$ contains staircase diagrams where the blocks $B_0'$ and $B_0''$ are not comparable. However, it is easy to check that Lemma \ref{L:D_case2_saturation} only fails for the diagrams where $B_0'=\{s_1\},B_0''=\{s_2\},B_0=\{s_3\}$ and $B_0$ is extremal.  Define the set
\begin{align*}
Z^{\circ}_{D3}(n)&:=\{\mcD\in Z_{D3}(n)\ |\ B_0'\prec B_0''\}
\end{align*}

Lemma \ref{L:D_case2_saturation} implies that $|Z_{D3}(n)|=2|Z_{D3}^{\circ}(n)|$ when $n\geq 4$ and $|Z_{D3}(3)|=2+2|Z_{D3}^{\circ}(3)|$.  Hence we focus on enumerating $Z_{D3}^{\circ}(n)$.  Note that staircase diagrams in $Z^{\circ}_{D3}(n)$ are not necessarily chains.  Moreover, if $\mcD\in Z_{D3}(n)$ is a chain, then $s_n$ may not be contained in an extremal block (See Example \ref{Ex:typeD_6}).  These issues are resolved by the next two lemmas.  Recall that $B_0',B_0''$ denote the unique blocks in $\mcD_{s_1},\mcD_{s_2}$ respectively and that $B_0'\prec B_0''$.

\begin{lemma}\label{L:D_case2.2_chain}
    Let $\mcD\in Z^{\circ}_{D3}(n)$.
    \begin{enumerate}
        \item If $B_0'\setminus\{s_1\}\subset B_0''$, then $\mcD\setminus\{B_0'\}$ is a chain.
        \item If $B_0''\setminus\{s_2\}\subset B_0'$, then $\mcD\setminus\{B_0''\}$ is a chain.
    \end{enumerate}
\end{lemma}

\begin{proof}
    It suffices to prove part (1).  Lemma \ref{L:D_case2_saturation} implies that there is a unique block $B_0\in\mcD_{s_3}$ between $B_0'$ and $B_0''$.  Since $$B_0'\setminus\{s_1\}\subset B_0''\setminus\{s_2\}\subset B_0,$$  the staircase diagram $\mcD\setminus\{B_0'\}$ has critical set $\{s_2,s_n\}$ and hence is a chain by Lemma \ref{L:chain_diagrams}.
\end{proof}

\begin{lemma}\label{L:D_case2.2_base_case}
     Let $\mcD\in Z^{\circ}_{D3}(n)$.  Then the following are equivalent:
     \begin{enumerate}
     \item $\mcD\setminus\{B_0'\}$ and $\mcD\setminus\{B_0''\}$ are both chains.
     \item $s_n$ is not contained in an extremal block of $\mcD$.
     \item $\mcD=(B_0'\prec B_0\prec B_0'')$ with $s_n\in B_0$.
     \end{enumerate}
\end{lemma}

\begin{proof}
Clearly (3) $\Rightarrow$ (1).  We will show that (1) $\Rightarrow$ (2) $\Rightarrow$ (3).  Since $\mcD\in Z^{\circ}_{D3}(n)$, we have $\mcD_{s_3}=(B_0'\prec B_0\prec B_0'')$.  Let $B_{\mcD}$ denote the unique block in $\mcD_{s_n}$.  By part (1), $B_0'$ is a minimum and $B_0''$ is a maximum.  If $B_{\mcD}$ is a minimum, then $\mcD\setminus\{B_0''\}$ has two minimums.  If $B_{\mcD}$ is a maximum, then $\mcD\setminus\{B_0'\}$ has two maximums.  Thus $B_{\mcD}$ cannot be extremal in $\mcD$. This proves (1) $\Rightarrow$ (2).

To prove (2) $\Rightarrow$ (3), it suffices by Lemma \ref{L:D_case2.2_chain} to assume that $\mcD\setminus\{B_0''\}$ is a chain.  Then $B_{\mcD}$ is maximal in $\mcD\setminus\{B_0''\}$ and $B_0'\prec B_0\prec B_{\mcD}$.  By part (2), $B_{\mcD}$ is not maximal in $\mcD$ and hence $B_{\mcD}\prec B_0''$ in $\mcD$.  But Lemma \ref{L:staircase_trees} implies $\mcD_{s_3}$ is saturated.  Hence $B_0=B_{\mcD}$ and $\mcD=\mcD_{s_3}$.
\end{proof}

One consequence of Lemmas \ref{L:D_case2.2_chain} and \ref{L:D_case2.2_base_case} is that there is a unique elementary diagram in $Z_{D3}(n)$ satisfying $B_0'\setminus\{s_1\}= B_0''\setminus\{s_2\}$.  Specifically, we define $$\mcH_n:=([s_1,s_{n-1}]\prec[s_3,s_n]\prec[s_2,s_{n-1}])$$ to be this unique element. For example, $\mcH_6$ is given by
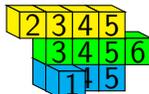
\begin{equation*}
    \begin{tikzpicture}[scale=0.35]
        \planepartitionD{6}{
            {0,1,1,2,0},
            {1,1,1,1,0},
            {0,1,1,1,1}}
    \end{tikzpicture}
\end{equation*}

\begin{lemma}\label{L:D_case2.2_maptoA}
If $\mcD\in Z_{D3}(n)\setminus\{\mcH_n\}$, then either $\mcD\setminus\{B_0'\}$ or $\mcD\setminus\{B_0''\}$ is an elementary chain of type $A_{n-1}$, but not both.
\end{lemma}
\begin{proof}Lemma \ref{L:D_case2.2_chain} implies that at least one of $\mcD\setminus\{B_0'\}$ or $\mcD\setminus\{B_0''\}$ is an elementary chain of type $A_{n-1}$.  If both are chains, then $\mcD=(B_0'\prec B_0\prec B_0'')$ with $s_n\in B_0$ by Lemma \ref{L:D_case2.2_base_case}.  Now only one of $(B_0'\prec B_0)$ or $(B_0\prec B_0'')$ can be elementary since $\mcD\neq \mcH_n$.
\end{proof}

Lemma \ref{L:D_case2.2_maptoA} implies that there is a well defined map
$$Z^{\circ}_{D3}(n)\setminus\{\mcH_n\}\rightarrow Z_A(n-1)$$
sending $\mcD$ to either $\mcD\setminus\{B_0'\}$ or $\mcD\setminus\{B_0''\}$.
\begin{example}\label{Ex:typeD2.2}
    The map $Z^{\circ}_{D3}(7)\setminus\{\mcH_n\}\rightarrow Z_A(6)$ sends the staircase diagrams
    
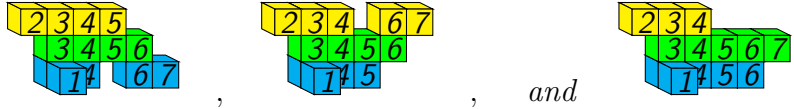
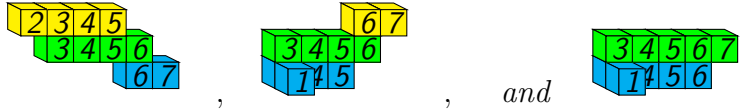
\begin{equation*}
        \begin{tikzpicture}[scale=0.35]
            \planepartitionD{7}{
                {1,1,0,1,2,0},
                {0,1,1,1,1,0},
                {0,0,1,1,1,1}}
        \end{tikzpicture}\quad, \quad
        \begin{tikzpicture}[scale=0.35]
            \planepartitionD{7}{
                {0,0,1,1,2,0},
                {0,1,1,1,1,0},
                {1,1,0,1,1,1}}
        \end{tikzpicture}\quad, \quad \text{ and }\quad
        \begin{tikzpicture}[scale=0.35]
            \planepartitionD{7}{
                {0,1,1,1,2,0},
                {1,1,1,1,1,0},
                {0,0,0,1,1,1}}
        \end{tikzpicture}
    \end{equation*}
    to the diagrams
    \begin{equation*}
        \begin{tikzpicture}[scale=0.35]
            \planepartitionD{7}{
                {1,1,0,0,0,0},
                {0,1,1,1,1,0},
                {0,0,1,1,1,1}}
        \end{tikzpicture}\quad, \quad
        \begin{tikzpicture}[scale=0.35]
            \planepartitionD{7}{
                {0,0,1,1,2,0},
                {0,1,1,1,1,0},
                {1,1,0,0,0,0}}
        \end{tikzpicture}\quad,\quad \text{ and } \quad
        \begin{tikzpicture}[scale=0.35]
            \planepartitionD{7}{
                {0,1,1,1,2,0},
                {1,1,1,1,1,0},
                {0,0,0,0,0,0}}
        \end{tikzpicture}
    \end{equation*}
    respectively.
\end{example}
If $\mcD\mapsto\mcD\setminus\{B_0'\}$ then $\mcD\setminus\{B_0'\}\in Z_A^-(n-1)$, while if $\mcD\mapsto\mcD\setminus\{B_0''\}$ then $\mcD\setminus\{B_0''\}\in Z_A^+(n-1)$.  Moreover, the number of diagrams $\mcD$ which map to $Z_A^-(n-1)$ and $Z_A^+(n-1)$ are equal.  The bijection between these two sets is to reverse the partial order on $\mcD$ and then swap the sizes of $B_0'$ and $B_0''$ (see first two diagrams in Example \ref{Ex:typeD2.2}).  Set
$$Z^{\circ,+}_{D3}(n):=\{\mcD\in Z^{\circ}_{D3}(n)\ |\ \mcD\setminus\{B_0''\}\in Z_A^+(n-1)\}.$$
Note that $Z^{\circ,+}_{D3}(n)$ includes the staircase diagram $\mcH_n$ and hence
\begin{equation}\label{Eq:typeD2.2}
|Z^{\circ}_{D3}(n)|=2|Z^{\circ,+}_{D3}(n)|-1.
\end{equation}
If $s_n$ is contained in a maximal block $B_{\mcD}$ of $\mcD\in Z^{\circ,+}_{D3}(n)$, then
$\mathfrak{G}_p(\mcD)\subseteq Z^{\circ,+}_{D3}(n+p)$ is well defined as per Definition \ref{D:Cat_gen}. Moreover, we have $$|\mathfrak{G}_p(\mcD)|=|\mathfrak{G}_p(\mcD\setminus\{B_0''\})|,$$ where we consider $\mcD\setminus\{B_0''\}$ as an elementary diagram of type $A$.
\begin{example}\label{Ex:typeD2.2_frak_gen1}
    If $\mcD\in Z^{\circ,+}_{D3}(7)$ is the top diagram depicted below, then $\mathfrak{G}_1(\mcD)$
    contains two diagrams as shown.
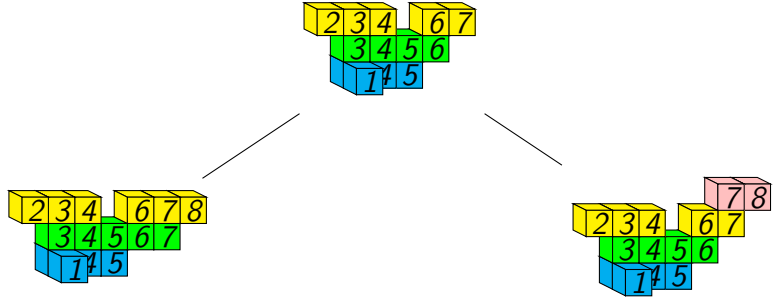
\begin{equation*}
\begin{tikzpicture}[level distance=25mm,
level 1/.style={sibling distance=75mm},
level 2/.style={sibling distance=30mm}]
    \node {
        \begin{tikzpicture}[scale=0.35]
            \planepartitionD{7}{
                {0,0,1,1,2,0},
                {0,1,1,1,1,0},
                {1,1,0,1,1,1}}
        \end{tikzpicture}}
        child { node {
        \begin{tikzpicture}[scale=0.35]
            \planepartitionD{8}{
                {0,0,0,1,1,2,0},
                {0,1,1,1,1,1,0},
                {1,1,1,0,1,1,1}}
        \end{tikzpicture}}}
        child { node {
        \begin{tikzpicture}[scale=0.35]
            \planepartitionD{8}{
                {0,0,0,1,1,2,0},
                {0,0,1,1,1,1,0},
                {0,1,1,0,1,1,1},
                {1,1,0,0,0,0,0}}
        \end{tikzpicture}}};
\end{tikzpicture}
\end{equation*}
\end{example}

If $s_n$ is not contained in an extremal block of $\mcD\in Z^{\circ,+}_{D3}(n)$, then $\mathfrak{G}_1(\mcD)$ is not defined as per Definition \ref{D:Cat_gen}.  In this case, we make the following adjustments.  Since $s_n$ is not contained in an extremal block of $\mcD\in Z^{\circ,+}_{D3}(n)$, Lemma \ref{L:D_case2.2_base_case} implies that $\mcD=(B_0'\prec B_0\prec B_0'')$ with $B_0=[s_3,s_n]$ and $B_0'=[s_1,s_{n-1}]$. Let
\begin{equation*}
    P'(B_0):=\{B\in P(B_0)\ |\ \text{$B$ is not adjacent to $B_0''$}\}
\end{equation*}
where $P(B_0)$ is given in Definition \ref{D:Cat_gen}.  Define
\begin{equation*}
    \mathfrak{G}_1(\mcD):=\{\mcD^0\}\cup \bigcup_{B\in P'(B_0)}\{\mcD^B\},
\end{equation*}
where $\mcD^0:=(B_0'\cup\{s_n\}\prec B_0\cup\{s_{n+1}\}\prec B_0'')$, and, for any $B\in P'(B_0)$,
\begin{equation*}
    \mcD^B:=\mcD\cup\{B\cup\{s_{n+1}\}\}
\end{equation*}
with the additional covering relation $B_0\prec B\cup\{s_{n+1}\}$.  It is easy to see that if $\mcD\in Z^{\circ,+}_{D3}(n+1)$, then $\mathfrak{G}_1(\mcD)\subseteq Z^{\circ,+}_{D3}(n+1)$ and $\mathfrak{G}_p(\mcD)\subseteq Z^{\circ,+}_{D3}(n+p)$.

\begin{example}\label{Ex:typeD2.2_frak_gen2}
    Consider the diagram $\mcH_5$.   Then $\mathfrak{G}_p(\mcH_5)$ for $p=1,2$ are given by the following diagrams:

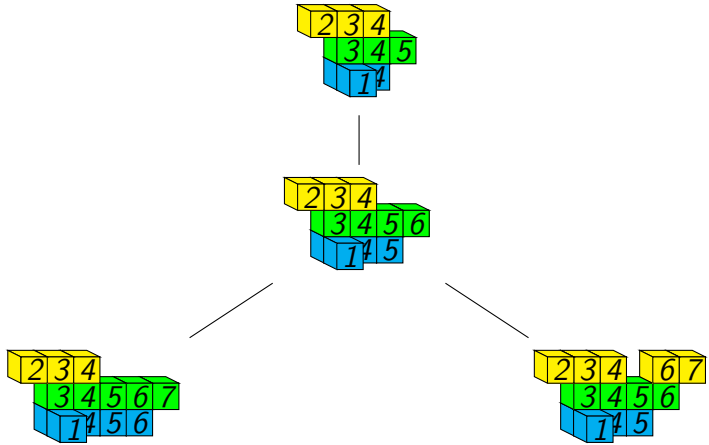
\begin{equation*}
\begin{tikzpicture}[level distance=23mm,
level 1/.style={sibling distance=10mm},
level 2/.style={sibling distance=70mm}]
    \node {
        \begin{tikzpicture}[scale=0.35]
            \planepartitionD{5}{
                {0,1,2,0},
                {1,1,1,0},
                {0,1,1,1}}
        \end{tikzpicture}}
        child { node {
        \begin{tikzpicture}[scale=0.35]
            \planepartitionD{6}{
                {0,1,1,2,0},
                {1,1,1,1,0},
                {0,0,1,1,1}}
        \end{tikzpicture}}
            child { node {
        \begin{tikzpicture}[scale=0.35]
            \planepartitionD{7}{
            {0,1,1,1,2,0},
            {1,1,1,1,1,0},
            {0,0,0,1,1,1}}
        \end{tikzpicture}}}
            child { node {
        \begin{tikzpicture}[scale=0.35]
            \planepartitionD{7}{
            {0,0,1,1,2,0},
            {0,1,1,1,1,0},
            {1,1,0,1,1,1}}
        \end{tikzpicture}
            }}};
\end{tikzpicture}
\end{equation*}
\end{example}


\begin{lemma}\label{L:D_case2.2_Catalan}
    For any $n\geq 3$, we have $|\mathfrak{G}_p(\mcH_n)|= \cc_p$.
\end{lemma}
\begin{proof}
The set $\mathfrak{G}_1(\mcH_n)$ contains a single diagram $$\mcD=([s_1,s_{n}]\prec[s_3,s_{n+1}]\prec[s_2,s_{n-1}]).$$  The definition of $\mathfrak{G}_1$ on $\mcD$ induces bijection between $\mathfrak{G}_p(\mcD)$ and $\mathfrak{G}_p(\mcD')$ where $\mcD'$ is the type $A_3$ elementary diagram $([s_1,s_2]\prec[s_2,s_3])$.  The lemma now follows from Proposition \ref{P:Catalan_gen}, part (2).
\end{proof}

\begin{proposition}\label{P:D_case2.2}
    $\displaystyle |Z^{\circ}_{D3}(n)|=-1+2\sum_{k=0}^{n-3}\cc_k$.
\end{proposition}
\begin{proof}
    It follows from Lemma \ref{L:Catalan_almostsurjective} that
    $$\bigcup_{k=3}^{n}\mathfrak{G}_{n-k}(\mcH_k)=Z^{\circ,+}_{D3}(n).$$
    Lemma \ref{L:D_case2.2_Catalan} implies that
    $$|Z^{\circ,+}_{D3}(n)|= \sum_{k=0}^{n-3}\cc_k.$$  The proposition now follows from Equation \eqref{Eq:typeD2.2}.
\end{proof}
\subsection{The generating series of type $D$}
Define the generating series
$$D_Z(t):=\sum_{n=3}^{\infty}dz_n\, t^n,$$
where $dz_n:=|Z_{D}(n)|$ denotes the number of elementary staircase diagrams of type $D_n$.
\begin{proposition}
The coefficient $dz_3=11$, and if $n\geq 4$ then $$\displaystyle dz_n=-2\cc_{n-3}+8\sum_{k=0}^{n-2}\cc_k.$$
Consequently
$$D_Z(t)= -3t^3-8t^2+\left(\frac{2t^4-2t^3+8t^2}{1-t}\right)\Cat(t).$$
\end{proposition}
\begin{proof}
Propositions \ref{P:D_case1}, \ref{P:D_case2.1}, and \ref{P:D_case2.2} imply that
\begin{align*}
    dz_3 &= |Z_{D1}(3)|+|Z_{D2}(3)|+|Z_{D3}(3)|\\
    &= |Z_{D1}(3)|+2|Z^{\circ}_{D2}(3)|+2|Z^{\circ}_{D3}(3)|+2\\
    &= 1+6+4\cc_0= 11.
\end{align*}
For $n\geq 4$, we have
\begin{align*}
    dz_n &= |Z_{D1}(n)|+2|Z^{\circ}_{D2}(n)|+2|Z^{\circ}_{D3}(n)|\\
    &= (4\cc_{n-2}-2\cc_{n-3})+\left(2+4\sum_{k=0}^{n-2}\cc_k\right)+\left(-2+4\sum_{k=0}^{n-3}\cc_k\right)\\
    &= -2\cc_{n-3}+8\sum_{k=0}^{n-2}\cc_k.
\end{align*}
Now
\begin{align*}
    D_Z(t)&= 11t^3-2t^3\sum_{n=1}^{\infty}\cc_n t^n+8t^2\sum_{n=2}^{\infty}\left(\sum_{k=0}^{n-2}\cc_k\right)t^n\\
    &= 11t^3-2t^3(\Cat(t)-1)+8t^2\left(\frac{\Cat(t)}{(1-t)}-1-2t\right).
\end{align*}
\end{proof}
Define the generating series
$$\overline{D}(t):=\sum_{n=3}^{\infty} \overline{d}_n\, t^n$$
where $\overline{d}_n$ denotes the number of staircase diagrams of type $D_n$ with full support.

\begin{proposition}\label{P:typeD_genseries_full}
The generating series $$\overline{D}(t)=\frac{\overline{A}(t)D_Z(t)}{t}-\frac{2t^2(\overline{A}(t)-t)}{1-t},$$ where $\overline{A}(t)$ is the generating series for the number of fully supported staircase diagrams of type $A$.
\end{proposition}

\begin{proof}
The proof is very similar to the proof of the type $A$ case in Proposition \ref{P:typeA_genseries_full}.  Let $\mcD$ be a staircase diagram of type $D_n$.  Let $s_k$ denote the critical point of $\mcD$ with largest index $k$, where $k> 2$.  If $\mcD$ has no such critical point, then Lemma \ref{L:D_crit_pt_bound} implies that $\mcD$ is either $([s_2,s_n]\prec[s_1,s_n])$ or $([s_1,s_n]\prec[s_2,s_n])$.  Otherwise we can write $\mcD$ as the union of $\mcD',\mcD''$ where we intersect each element of $\mcD$ with supports $\{s_1\}\cup[s_2,s_k]$ and $[s_k,s_n]$ respectively.  It is easy to see that $\mcD'$ is an elementary staircase diagram of type $D_k$ with $s_k$ a critical point, and $\mcD''$ is a fully supported staircase diagram of type $A_{n-k+1}$.  Hence
$$\overline{d}_n=2+\sum_{k=3}^{n}\overline{a}_{n-k+1}(dz_{k}-2).$$
This implies that
\begin{align*}
\overline{D}(t)&=\sum_{n=3}^{\infty}\left(2+\sum_{k=3}^n(\overline{a}_{n-k+1}(dz_{k}-2))\right)\, t^n\\
&=\sum_{n=3}^{\infty}\left(\sum_{k=3}^n \overline{a}_{n-k+1}dz_{k}\right)\, t^n -
2\sum_{n=3}^{\infty}\left(-1+\sum_{k=1}^{n-2} \overline{a}_k\right)\, t^n\\
&=\frac{\overline{A}(t)D_Z(t)}{t}-\frac{2t^2(\overline{A}(t)-t)}{1-t}.
\end{align*}
\end{proof}
Recall from the introduction that
$$D(t):=\sum_{n=3}^{\infty}d_n\, t^n,$$
where $d_n$ denotes the number of staircase diagrams of type $D_n$.

\begin{theorem}
The generating series $$D(t)=-2t-3t^2+(2t-t^2+t^3)A(t)+(tA(t)+1)\overline{D}(t),$$ where $A(t)$ is the generating series for the number of staircase diagrams of type $A$.
\end{theorem}
\begin{proof}
We break staircase diagrams of type $D_n$ into four categories.

Diagrams with support contained in $[s_2,s_n]$ or $[s_1,s_n]$ are diagrams of type $A_{n-1}$.  Hence the sub-generating series for diagrams of this type in $D_n$ is $$2t(A(t)-1-2t)-t^2(A(t)-1).$$

Diagrams where both $s_1$ and $s_2$ are in the support and $s_3$ is not in the support are of type $A_{n-3}$.  Hence their generating series is $t^3A(t)$.

Diagrams with $s_1,s_2,s_3$ in the support, but which are not fully supported, are a disjoint union of a staircase diagram of type $A_k$ and a fully supported diagram of type $D_{n-k-1}$.  Hence their generating series is $tA(t)\overline{D}(t)$.

Diagrams not in the cases above are simply fully supported diagrams of type $D_n$ and have generating series $\overline{D}(t)$.
Thus
$$D(t)=((2t-t^2)A(t)-2t-3t^2)+t^3A(t)+tA(t)\overline{D}(t)+\overline{D}(t)$$
\end{proof}

\bibliographystyle{amsalpha}
\bibliography{palindromic}

\end{document}